\documentclass[reqno, 12pt]{amsart}

\usepackage{geometry}
\usepackage[dvipsnames]{xcolor}
\usepackage[shortlabels]{enumitem}
\usepackage{lineno}
\usepackage{marginnote}
\usepackage[colorinlistoftodos]{todonotes}

\usepackage[open,openlevel=2,atend]{bookmark}

\usepackage[abbrev,msc-links,backrefs]{amsrefs}
\usepackage{doi}

\renewcommand{\PrintDOI}[1]{\doi{#1}}

\usepackage[T1]{fontenc}
\usepackage[english]{babel}
\usepackage[utf8]{inputenc}

\usepackage{hyperref}
\usepackage{url}
\hypersetup{%
    colorlinks,%
    linkcolor = {red!60!black},%
    citecolor = {green!60!black},%
    urlcolor = {blue!60!black}%
}

\usepackage{changes}
\definechangesauthor[name=Giovanne, color=blue]{gio}

\usepackage{graphicx}
\graphicspath{{../figures/}}

\usepackage{tikz}
\usetikzlibrary{arrows,calc,positioning,decorations.pathmorphing,decorations.pathreplacing,decorations.markings}

\usetikzlibrary{arrows.meta}

\definecolor{BrickRed}{HTML}{B6321C}
\definecolor{BurntOrange}{HTML}{F7921D}
\definecolor{Bluezao}{HTML}{003F88}
\definecolor{RedViolet}{HTML}{A1246B}
\definecolor{Red}{HTML}{ED1B23}
\definecolor{Sepia}{HTML}{671800}
\definecolor{RawSienna}{HTML}{974006}
\definecolor{BurntOrange}{HTML}{F7921D}
\definecolor{green}{RGB}{126,198,54}

\tikzstyle{vertex}=[fill=white, draw=black, shape=circle, inner sep=2.2pt]
\tikzstyle{vertex mini}=[fill=black, draw=black, shape=circle, inner sep=1pt]
\tikzstyle{vertex exp}=[fill=white, draw=black, shape=circle, inner sep=1pt]

\tikzstyle{cluster}=[fill=white, draw=black, shape=circle, inner sep=10pt]
\tikzstyle{edge}=[->, draw=black, ultra thick, line width=1.2pt, opacity=0.8]
\tikzstyle{edge superregular}=[draw=gray, line width=19pt, opacity=0.3]
\tikzstyle{edge skewed}=[draw=BrickRed, line width=19pt, opacity=0.2]
\tikzstyle{edge path}=[-, draw=Bluezao, ultra thick, line width=1.2pt, decoration={markings, mark=at position 0.5 with {\arrow{>}}}]
\tikzstyle{cobra}=[->, BrickRed, ultra thick, line width=2pt]
\tikzstyle{set}=[fill=white, draw=black, shape=circle, inner sep=10pt]
\tikzstyle{set 2}=[fill=white, draw=black, dotted, shape=circle, inner sep=0.5cm]
\tikzstyle{set 3}=[fill=white, draw=black, thick, shape=circle, inner sep=0.7cm]
\tikzstyle{set 4}=[fill=gray,opacity=0.3, draw=black, thick, shape=circle, inner sep=0.7cm]

\usepackage{amsmath}
\usepackage{amssymb}
\usepackage{amscd}
\usepackage{amsthm} 
\usepackage{dsfont}
\usepackage{mathtools}

\newtheorem{thm}{Theorem}[section]
\newtheorem{lemma}[thm]{Lemma}

\newtheorem{prop}[thm]{Proposition}
\newtheorem{fact}[thm]{Fact}

\theoremstyle{definition}
\newtheorem{defn}[thm]{Definition}

\newtheorem{ques}[thm]{Question}
\theoremstyle{remark}

\newcommand{\floor}[1]{{\left\lfloor #1 \right\rfloor}}
\newcommand{\ceil}[1]{{\left\lceil #1 \right\rceil}}
\newcommand{\defi}[1]{\emph{#1}}

\newcommand{\sepline}{\noindent\textcolor{gray}{\rule{\linewidth}{0.1pt}}\vspace{0.2em}}

\def\Root{\textsc{Root}}
\def\NN{\mathbb N}
\def\semidegree{\delta^{0}}
\def\RN{\textsc{RN}}
\def\Exp{\mathbb{E}}
\def\P{\mathbb{P}}
\let\eps=\varepsilon
\let\theta=\vartheta
\let\phi=\varphi

\def\cP{\mathcal{P}}
\def\cR{\mathcal{R}}

\def\qand{\quad\text{and}\quad}

\renewcommand{\subset}{\subseteq}

\begin{document}

\title{Semidegree threshold for spanning trees in oriented graphs}
\date{\today}

\author[P. Araújo]{Pedro Araújo}
\address{Departamento de Matemática, Universidade Federal de Pernambuco, Recife, Brasil}
\email{pedrocampos.araujo@ufpe.br}

\author[G. Santos]{Giovanne Santos}
\address{Departamento de Ingeniería Matemática,
Universidad de Chile, Santiago, Chile}
\email{gsantos@dim.uchile.cl}

\author[M. Stein]{Maya Stein}
\address{\parbox{5.5in}{Departamento de Ingeniería Matemática y Centro de Modelamiento
Matemático (CNRS IRL2807), Universidad de Chile, Santiago, Chile\vspace{0.2cm}}}
\email{mstein@dim.uchile.cl}

\thanks{Pedro Ara\'ujo received funding from ANID grant CMM Basal FB210005 and was supported by the grant 23-06815M of the Grant Agency of the Czech Republic. 
Giovanne Santos was supported by ANID Becas/Doctorado Nacional 21221049. Maya Stein was supported by ANID Regular Grant 1221905 and by ANID grant CMM Basal FB210005. 
This project has also received funding from the European Union’s Horizon 2020 research and innovation programme under the Marie Skłodowska-Curie grant agreement No.~101007705 (RandNET)}


\begin{abstract}
    We show that for all $\gamma > 0$ and $\Delta \in \NN$, there is
    some $n_0$ such that, if $n \geq n_0$, then every oriented graph
    on $n$ vertices with minimum semidegree at least $(3/8 + \gamma)n$ contains
    a copy of each oriented tree on $n$ vertices with maximum degree at most $\Delta$. This is asymptotically best possible.
\end{abstract}

\maketitle

\section{Introduction}
\label{sec:intro}

A classical result of Dirac~\cite{dirac} states that every graph on $n \geq 3$ vertices with minimum degree at 
least $n/2$ contains a Hamilton cycle, i.e.~a cycle on $n$ vertices. A closely related result is
the famous theorem of Koml\'{o}s, S\'{a}rk\"{o}zy and Szemer\'{e}di~\cite{kss1} which states
that, for any $\gamma >0$ and $\Delta \in \NN$, there exists a~$n_0$ such that every graph
on $n \geq n_0$ vertices
with minimum degree at least~${(1/2+\gamma)n}$ contains every spanning tree
with maximum degree at most $\Delta$. The bound on the maximum degree  of the tree  was later improved by the same authors  to  $O(n/\log n)$~\cite{kss2}.
In this paper, we investigate Dirac-type conditions for the embedding of bounded
degree trees in digraphs. 

One of the earliest Dirac-type results for digraphs is due to 
Ghouila-Houri~\cite{gh}, who showed that every strongly connected digraph~$D$ on $n$ vertices
with~${\semidegree(D) \geq n/2}$ has a directed Hamilton cycle (for all notation, see Section~\ref{sec:preliminaries}). A generalization of this theorem, due to DeBiasio and Molla~\cite{dm} and 
DeBiasio, K\"{u}hn, Molla, Osthus and Taylor~\cite{dkmot} states that $\semidegree(D) \geq n/2 + 1$ suffices to guarantee all   orientations of the Hamilton cycle in an $n$-vertex digraph $D$, if $n$ is  sufficiently large, and moreover, the term $+1$ is only needed for the  antidirected Hamilton cycle. With respect to spanning trees, 
Mycroft and Naia~\cite{tassio} showed that
for any $\gamma > 0$ and $\Delta \in \NN$ there exists $n_0$ such
that any digraph $D$ on $n \geq n_0$ vertices with $\semidegree(D) \geq (1/2 + \gamma)n$
contains every oriented tree $T$ on $n$ vertices with $\Delta(T) \leq \Delta$, which is reproved in this work. Kathapurkar and Montgomery~\cite{km} improved this, relaxing the maximum degree bound for the tree to $O(n/\log n)$. In conclusion, both for Hamiltonicity and for spanning tree universality, the bounds on the minimum degree in the graph setting correspond to the bounds on the minimum semidegree in the digraph setting (if we disregard the $+1$ term for antidirected Hamilton cycles).

This changes in the setting of  oriented graphs. Kelly, K\"{u}hn and Osthus~\cite{kelly}
showed that for any~${\gamma > 0}$, every sufficiently large oriented graph $G$ on $n$ vertices
with~${\semidegree(G) \geq (3/8+\gamma)n}$ contains a directed Hamilton
cycle. An exact result for sufficiently
large oriented graphs was later proved by Keevash, K\"{u}hn and
Osthus~\cite{exact}, and Kelly~\cite{kelly2}   proved
asymptotically optimal results for Hamilton cycles with any orientation. This was recently improved by 
Wang, Wang and Zhang~\cite{wwz25} who showed that every
sufficiently large oriented graph with minimum semidegree at least $\ceil{(3n - 1)/8}$
contains every possible orientation of a Hamilton cycle.

To our knowledge, until now no corresponding results are known for the spanning tree problem in oriented graphs.
We show that in this setting, the minimum semidegree threshold for containment of bounded-degree oriented
tree  asymptotically coincides with the threshold for oriented Hamilton cycles.

\begin{thm}
  \label{thm:main_oriented}
  For every $\gamma > 0$ and $\Delta \in \NN$, there exists $n_0 \in \NN$
  such that the following holds for every $n \geq n_0$. If $G$ is an oriented graph
  on $n$ vertices with ${\semidegree(G) \geq (3/8 + \gamma)n}$, and $T$ is an oriented
  tree on $n$ vertices with $\Delta(T) \leq \Delta$, then $G$ contains a copy of $T$.
\end{thm}

The following construction due Kelly~\cite[Proposition 4]{kelly2}, which improves  earlier
constructions of H\"aggkvist~\cite{hag} and Keevash, K\"{u}hn, and Osthus~\cite{exact}, 
shows that the bound on the minimum semidegree in
Theorem~\ref{thm:main_oriented} is asymptotically tight. Let $n=8k+4$, for some~${k \in \NN}$.
Consider the oriented graph $G$ on $n$ vertices such that 
the vertex set~$V(G)$ is partitioned into
four sets~$X,Y,W$ and $Z$, each with exactly $2k+1$ vertices. The sets~$X$ and~$Z$
each induce a regular tournament. All arcs from~$X$ to $W$, from $W$ to $Z$,
from~$Z$ to~$Y$, and from $Y$ to $X$ are present in $G$. Finally, the pair $(Y,W)$ induces 
a regular bipartite tournament (see Figure~\ref{fig:construction}). Note that $G$ has minimum semidegree exactly 
$3n/8-1/2$, 
and that it does not contain any antidirected path on more than $3n/4$ vertices. In particular,
the graph $G$ does not contain the antidirected Hamilton path. Recently, this
construction was extended to all $n$ by
Wang, Wang and Zhang~\cite[Proposition~1.2]{wwz25}.

\begin{figure}[ht!]
    \centering
    \begin{tikzpicture}[scale=0.7]
  \node [style=set] (W) at (0, 3) {};
	\node [style=set] (Z) at (3, 0) {};
	\node [style=set] (Y) at (6, 3) {};
	\node [style=set] (X) at (3, 6) {};

  \node [vertex mini] (a0) at (3,6.4) {};
  \node [vertex mini] (a1) at (2.6,6) {};
  \node [vertex mini] (a2) at (3.4,6) {};
  \node [vertex mini] (a3) at (2.8,5.6) {};
  \node [vertex mini] (a4) at (3.2,5.6) {};
  
  \begin{scope}[yshift=-6cm]
  \node [vertex mini] (b0) at (3,6.4) {};
  \node [vertex mini] (b1) at (2.6,6) {};
  \node [vertex mini] (b2) at (3.4,6) {};
  \node [vertex mini] (b3) at (2.8,5.6) {};
  \node [vertex mini] (b4) at (3.2,5.6) {};
  \end{scope}

%

  \draw[thin, arrows = {-Stealth[length=3pt]}] (a0) -- (a1);
  \draw[thin, arrows = {-Stealth[length=3pt]}] (a0) -- (a3);
  \draw[thin, arrows = {-Stealth[length=3pt]}] (a2) -- (a0);
  \draw[thin, arrows = {-Stealth[length=3pt]}] (a4) -- (a0);
  \draw[thin, arrows = {-Stealth[length=3pt]}] (a1) -- (a2);
  \draw[thin, arrows = {-Stealth[length=3pt]}] (a1) -- (a3);
  \draw[thin, arrows = {-Stealth[length=3pt]}] (a3) -- (a4);
  \draw[thin, arrows = {-Stealth[length=3pt]}] (a3) -- (a2);
  \draw[thin, arrows = {-Stealth[length=3pt]}] (a2) -- (a4);
  \draw[thin, arrows = {-Stealth[length=3pt]}] (a4) -- (a1);

  \draw[thin, arrows = {-Stealth[length=3pt]}] (b0) -- (b1);
  \draw[thin, arrows = {-Stealth[length=3pt]}] (b0) -- (b3);
  \draw[thin, arrows = {-Stealth[length=3pt]}] (b2) -- (b0);
  \draw[thin, arrows = {-Stealth[length=3pt]}] (b4) -- (b0);
  \draw[thin, arrows = {-Stealth[length=3pt]}] (b1) -- (b2);
  \draw[thin, arrows = {-Stealth[length=3pt]}] (b1) -- (b3);
  \draw[thin, arrows = {-Stealth[length=3pt]}] (b3) -- (b4);
  \draw[thin, arrows = {-Stealth[length=3pt]}] (b3) -- (b2);
  \draw[thin, arrows = {-Stealth[length=3pt]}] (b2) -- (b4);
  \draw[thin, arrows = {-Stealth[length=3pt]}] (b4) -- (b1);

  \draw[ultra thick,arrows = {-Stealth}] (X) -- (W);
  \draw[ultra thick,arrows = {-Stealth}] (W) -- (Z);
  \draw[ultra thick,arrows = {-Stealth}] (Z) -- (Y);
  \draw[ultra thick,arrows = {-Stealth}] (Y) -- (X);

  \draw[ultra thick, arrows = {-Stealth[harpoon]}] (0.7,3.1) -- (5.3,3.1);
  \draw[ultra thick, arrows = {-Stealth[harpoon]}] (5.3,2.9) -- (0.7,2.9);
  
  \begin{scope}[node distance=0.8cm]
  \node[left of=X] {$X$};
  \node[left of=W] {$W$};
  \node[right of=Y] {$Y$};
  \node[right of=Z] {$Z$};
  \end{scope}

\end{tikzpicture}
    \caption{An oriented graph with minimum degree $3n/8-1/2$ without an antidirected Hamilton path.}
    \label{fig:construction}
\end{figure}

Two open problems arise from our work. The first is determining  the exact threshold for
an oriented graph to contain bounded degree trees, and the second is determining  the asymptotic threshold for trees with
maximum degree that grows with $n$.
In particular, it seems natural to ask if the bound on the maximum
degree of the tree in Theorem~\ref{thm:main} can be relaxed to $O(n/\log n)$,
as in the works of Koml\'{o}s, S\'{a}rk\"{o}zy and Szemer\'{e}di~\cite{kss2},
and Kathapurkar and Montgomery~\cite{km}.

However, as Rozho\v{n}
and \v{S}ámal~\cite{rs26}
pointed out, for sufficiently large $n$, the $k$th
power~$C_n^{k}$ of a directed cycle on $n$
vertices, where~${k=(3/8 + \alpha)n}$
and $0 < \alpha < 1/8$, does
not contain the oriented tree~$T$ in which the root
has out-degree~$\sqrt{n}-1$ and each of its out-neighbours
has out-degree~$\sqrt{n}$. Indeed, $C_n^{k}$ is a regular
digraph with in- and out-degree~$k$,
and every vertex~$x$ in $C_n^{k}$ has the property
that~$|N^{+}(N^{+}(x))| = 2k$.
In contrast, $T$ contains a vertex $u$ such that $|N^{+}(N^{+}(u))| = n-1$.
Since $2k < n-1$ for $n$ sufficiently
large, $C_n^{k}$ cannot contain a copy of $T$.
This naturally leads to the following question.

\begin{ques}
    \label{ques:ASS_conj}
    Can the bounded maximum degree assumption on the tree in Theorem~\ref{thm:main_oriented}
    be replaced by the weaker assumption $\Delta(T) \leq f(n)$, for some function $f(n) \rightarrow \infty$?
\end{ques}

Since our proof of Theorem~\ref{thm:main_oriented}
crucially relies on the assumption that the tree has
bounded maximum degree,
extending our approach to Question~\ref{ques:ASS_conj} would require
additional ideas.

\subsection{Robust expansion} 
In order to prove Theorem~\ref{thm:main_oriented}, we will establish a version of it for
robust out-expanders. Given $\nu, \tau > 0$, we say that a digraph $D$ on~$n$ vertices
is a \defi{robust $(\nu,\tau)$-out-expander} if for every set $S \subseteq V(D)$
with~${\tau n \leq |S| \leq (1-\tau)n}$,  at least $|S| + \nu n$ vertices each receive at least $\nu n$
arcs from $S$. The notion of robust out-expanders was introduced
by K\"{u}hn, Osthus and Treglown~\cite{hamiltonian-degree} to prove Theorem~\ref{thm:hamilton} below
and has found many applications in extremal problems in digraphs. In the statement of the following
theorem and in the rest of the paper, we use the notation~$x \ll y$ 
to indicate that
there is an increasing  function~$f$ such that,
for any given $y$, whenever
we choose~${x \leq f(y)}$, all calculations
and inequalities
involving these constants in the sequence are valid.

\begin{thm}[\cite{hamiltonian-degree}]
  \label{thm:hamilton}
  Let $1/n \ll \nu \leq \tau \ll \gamma < 1$. Let $D$ be a a robust $(\nu,\tau)$-out-expander 
  on $n$ vertices with~${\semidegree(D) \geq  \gamma n}$. Then~$D$ contains a
  directed Hamilton cycle.
\end{thm}

We prove that in the same setting, $D$ also contains all bounded degree oriented spanning trees.

\begin{thm}
  \label{thm:main}
  Let $1/n \ll \nu \leq \tau \ll \gamma, \Delta^{-1}$.
  Let $D$ be a a robust $(\nu,\tau)$-out-expander 
  on~$n$ vertices with~${\semidegree(D) \geq  \gamma n}$, and let~$T$ be an oriented tree on~$n$
  vertices with~${\Delta(T) \leq \Delta}$. Then $D$ contains a copy of $T$.
\end{thm}

Now, Theorem~\ref{thm:main_oriented} follows immediately from Theorem~\ref{thm:main} and an application
of part~(b) of the following known result.

\begin{lemma}[\cite{ko}]\label{lem:outex}
    Let $1/n \ll \nu \ll \tau \leq \gamma$ and let $D$ be a digraph on $n$ vertices. If
    \begin{enumerate}[(a)]
     \item  $D$
    satisfies~${\semidegree(D) \geq (1/2 + \gamma)n}$, 
    or
          \item $D$
    is an oriented graph  with~${\semidegree(D) \geq (3/8 + \gamma)n}$, 
    \end{enumerate}
  then~$D$ is
    a robust $(\nu,\tau)$-out-expander.
\end{lemma}

Clearly, Theorem~\ref{thm:main} 
together with part (a) of Lemma~\ref{lem:outex} implies 
 Mycroft and Naia's above-mentioned result~\cite{tassio} on bounded degree oriented spanning trees in digraphs.

\subsection{Organisation} The paper is organised as follows.
Section~\ref{sec:overview} gives an overview of the proof of
Theorem~\ref{thm:main}. Section~\ref{sec:preliminaries}
introduces notation and auxiliary results. Section~\ref{ssec:diregularity}
contains the statements of the Diregularity Lemma and of the Blow-up Lemma.
In Section~\ref{sec:expansion}, we explore the notion of robust expansion and
prove some auxiliary results. Section~\ref{sec:random_walks} focuses on random
walks in digraphs, which are then used in Section~\ref{sec:semi-random} to
build a semi-random algorithm to assign vertices of the tree to clusters of a
reduced digraph. Finally, in Section~\ref{sec:proof_main}, we present the
proof of Theorem~\ref{thm:main}.

\section{Overview}
\label{sec:overview}

For better readability, we first sketch a proof for a variant of Theorem~\ref{thm:main} where $T$ is
only almost-spanning. Then we present the adjustments that are necessary so
the proof works for spanning trees.

\subsection{Proof for almost spanning trees}

We use the Digraph Regularity Lemma.
For an appropriate choice of
constants $\varepsilon\ll d \ll \nu\leq \tau \ll \gamma$, we find an $\varepsilon$-regular
partition $\{V_0,\dots,V_k\}$ of $V(D)$  
(see Section \ref{ssec:diregularity}). The corresponding reduced digraph $R$ 
inherits the minimum semidegree and
expansion of $D$ (see Lemma~\ref{lemma:reduced}).

As it is typical in applications of the Regularity Lemma, we need to find an assignment~${\phi:V(T) \rightarrow V(R)}$ that distributes
the vertices of $T$ in a balanced way among the vertices of~$R$. To find such $\varphi$, we use a randomised approach.
Consider, as a toy example, the case of $T$ being an oriented
path~${(v_1,v_2,\dots,v_n)}$. We start by
choosing $\phi(v_1)\in V(R)$ uniformly at random. Then, for $1 \leq t \leq n-1$, we choose 
$\phi(v_{t+1})$ uniformly at random either from the in- or from the out-neighbourhood of $\phi(v_{t})$, according to the orientation of~$T$. This gives a
random walk in $R$ that follows the pattern of the oriented path. To the best of our knowledge, this idea was first used by Mycroft and Naia~\cite{tassio}
in a slightly different setting. Since we require $\phi$ to be well-balanced, we keep track of the visitation time~$X_t$ of this walk
(i.e, $|\phi^{-1}(X)|$ at time~$t$) for $X\in V(R)$, which we expect to be roughly~$n/k$ once the
random walk mixes into a uniform distribution. 

Now the \textit{cherry property} (see Definition \ref{def:cherry}) becomes important.
It can be viewed as a directed analogue of being connected and non-bipartite in graphs, which are properties known to relate to
mixing of (undirected) random  walks.
It turns out that if $H$ is a regular digraph on $k$ vertices that possesses the
cherry property  then the
variable $X_t$ is w.h.p.~close to a uniform sampling in $V(H)$ as long as $t$ is large enough compared
to $k$ (see Proposition~\ref{proposition:mixing}). 
We also show that any robust expander
with non-vanishing minimum semidegree, as the one from Theorem \ref{thm:main}, has a spanning regular subdigraph having the cherry property (see Proposition~\ref{proposition:regular-cherry-subgraph}). In particular, this yields a regular subdigraph of~$R$ with the cherry property. 
Consequently, if performed on this subdigraph, the random walk mixes and w.h.p.~we obtain a well balanced assignment~$\phi$. 
We then complete the embedding~$T$
by greedily following~$\phi$. The properties of the regular partition guarantee that this greedy embedding works.

Now, if $T$ is an arbitrary oriented tree with bounded maximum degree (not necessarily a path), we modify our randomised approach as follows. We choose an arbitrary root $r$ of~$T$ and fix a top-bottom ordering $v_1,\dots,v_n$ of~$V(T)$. We start by assigning $v_1$ uniformly at random to a vertex of~$V(R)$. Then, for $1 \leq t \leq n-1$, we choose uniformly at random an out-neighbour $V^{+}$ and in-neighbour $V^{-}$ of~$\phi(v_{t})$. We assign all out-neighbours of $v_t$ to $V^{+}$ and all in-neighbours to $V^{-}$.
A martingale analysis shows that we can find a well-balanced assignment $\phi$ using this approach.

\subsection{Proof for  spanning trees}

To embed spanning oriented trees we must overcome three problems.  First, we have to use the Blow-Up lemma in order to fill up the clusters, second, we have to incorporate the exceptional vertices, and third, we need to find a perfectly balanced assignment. We will discuss each of these issues separately.

\subsubsection{Using the Blow-Up Lemma}

The greedy embedding procedure discussed above only works if the number of unused vertices in each cluster $V_i$
is much larger than $\eps|V_i|$, which is no longer the case when $T$ is spanning. To solve this problem we invoke the Blow-up Lemma, in a version due
to Csaba (see Lemma \ref{lemma:blow_up}).

To be able to apply the Blow-up Lemma, we need to make two modifications to our approach.
First, we find a directed Hamilton cycle $C$ in $R$, which exists by Theorem~\ref{thm:hamilton},
and relocate some vertices to $V_0$ to make the arcs of $C$ super-regular. Second,
we must ensure that many designated special arcs of $T$ are mapped by $\phi$ to the arcs of $C$ in
a balanced way. We call this the \emph{absorbing property}.
These special arcs are chosen according to the structure of $T$.
We distinguish three cases, depending
on whether $T$ has many leaves, many constant length {\it bare paths} (i.e.~paths on degree-2 vertices of $T$) of the same orientation,
or many {\it switches} (sources or sinks of degree two). We only explain here the many-paths case, as it contains the main ideas of the proof.
Also, we will assume that each of these paths starts with an out-vertex.
In this case, the special arcs are the arcs of selected bare paths.

To obtain an assignment with the absorbing property, we replace the random greedy assignment used in the almost-spanning case
with a {\it semi-random assignment}~(see Definition~\ref{def:semi_random}),
which is defined as follows. We perform the random greedy process described above, 
only that  whenever
the starting vertex of a special path is assigned, we deterministically assign the rest of the path only using edges of $C$. 
In Proposition~\ref{prop:sr_assignment},
we show that w.h.p.~this semi-random assignment $\phi$ satisfies the absorbing property and also gives each cluster 
roughly the same amount of vertices of~$T$.

Moreover, we can show that $\phi$ assigns roughly the same number of
starting vertices of the special paths to each cluster of $R$. This  property will be important in the next subsection. 

\subsubsection{Incorporating the exceptional vertices}

The second issue we need to address in order to embed spanning trees is that we must incorporate
the exceptional vertices from $V_0$ into~$\phi$. This will be done in Section~\ref{ssec:exceptional_vertices}, using
skewed-traverses (see Figure~\ref{fig:skewed}), which were introduced by Kelly~\cite{kelly}. For each exceptional
vertex $v \in V_0$ we find two clusters~$V_{i^{+}}$ and~$V_{i^-}$ such that $v$ sends many arcs to
$V_{i^{+}}$ and receives many arcs from $V_{i^-}$. These clusters exists because of the high minimum semidegree of $D$. 

Moreover, as established in the previous subsection, many special paths start at $V_{i^-}$. Let $P=(v_0,\dots,v_{\ell-1})$ be one of them. 
We incorporate $v$ by reassigning $v_1$ to $v$ and $v_2$ to~$V_{i^{+}}$. However,
note that this modification is only valid if $V_{i^{-}+3}$ is in the out-neighbourhood of~$V_{i^{+}}$; otherwise the
reassignment would violate the arc-preserving property of $\phi$. Since this condition need not always hold,
we use skewed-traverses to adjust the assignment of
some vertices of $P$ (see Figure~\ref{fig:incorp_bare}). This method
works only when $\ell$ is a multiple of $|V(R)|$ and, therefore, we can incorporate only $\xi n \ll \eps n$ vertices (while $V_0$ can have up to $\eps n$ vertices).

To solve this issue, we use an idea that has also appeared in~\cites{kelly,kelly2,taylor}. We split $T$ into two linear
subtrees $T_A$ and $T_B$ (see Lemma~\ref{lemma:split_tree}). We also split $D$
into two subdigraphs $D_{A'}$ and $D_{B'}$ such that both remain robust expanders with linear
minimum semidegree (see Lemma~\ref{lemma:partition}). We then apply the Diregularity Lemma
to $D_B$ with parameter $\eps$, obtaining an exceptional set $V_0$ with at most $\eps n$ vertices.
Next, we apply the Diregularity Lemma to $D[A \cup V_0]$, this time with parameter $\xi$ instead
of $\eps$, yielding an exceptional set $U_0$ with at most $\xi n$ vertices. Finally,
we define $D_A = D_{A'} \setminus U_0$ and $D_B = D_{B'}\cup U_0$. In this way, and  at the cost of having to connect $T_A$
and $T_B$ later, we have no exceptional vertices in $D_A$ and only a very small set $U_0$ of exceptional vertices in $D_B$, which
we can then incorporate using the method of skewed-traverses. The steps outlined above are carried out
in Sections~\ref{ssec:cases}, \ref{ssec:preparing_D},~\ref{ssec:exceptional_vertices}.

\subsubsection{Obtaining a perfectly balanced assignment}
Finally, the last issue in embedding spanning trees is obtaining a perfectly balanced assignment $\phi$.
We do this by modifying $\phi$ via skewed-traverses. This balancing step is very similar to the
incorporation of exceptional vertices and is presented in Section~\ref{ssec:balanced_assignment}.
Combined with the properties of~$\phi$, this procedure allows us to balance~$\phi$ without
significantly disturbing the absorbing property.

With all issues resolved, we obtain a perfectly balanced assignment $\phi$ that incorporates all the exceptional
vertices and still satisfies the absorbing property. We can now apply the Blow-up Lemma
to embed $T_B$ into $D_B$. Now, we need to embed $T_A$ into $D_A$ and connect the two pieces. 
We proceed in the same way as for $T_B$.
We find a perfectly balanced assignment $\phi_A$ of $T_A$ to the
corresponding reduced graph of $D_A$.
The only difference is that we ensure the root of $T_A$ is assigned
to a cluster that receives (or sends) many arcs to the vertex where the root of $T_B$ was embedded.
We can then apply the Blow-up Lemma to embed $T_A$ into $D_A$, finishing the embedding of $T$.

\section{Preliminaries}
\label{sec:preliminaries}

\subsection{On sequences of numbers}
\label{ssec:numbers}

Given a number $n \in \NN$, we denote by $[n]$ 
the set~${\{1,\dots,n\}}$. We will
use the following equality. 
\begin{prop}
  \label{prop:cauchy}
  For any two sequences of numbers $(a_i)_{i\in{[\ell]}}$
  and~$(b_i)_{i\in[\ell]}$ we have
  \begin{equation*}
  	\label{eq:Cauchy-Schwartz}
  	\left( \sum_{i\in[\ell]} a_i b_i \right)^2 = \sum_{i\in[\ell]}a_i^2\sum_{i\in[\ell]}b_i^2 - \frac{1}{2} \sum_{i,j\in[\ell]}\big(a_ib_j - a_jb_i \big)^2.
  \end{equation*} 
\end{prop}

\subsection{On probability and concentration inequalities}
\label{ssec:probability}

We say that an event depending on $n$ occurs \defi{with high probability} if its
probability tends to $1$ as $n$ goes to infinity. Given a random variable $X$,
we write $\Exp[X]$ for its expected value. We say that $X$ follows a binomial
distribution with parameters $n$ and $p$ if it counts the number of successes
in~$n$ independent experiments, each of which has probability $p$ of success.
We use the following Chernoff-type bound (see, for instance,
Corollary 2.3 of~\cite{jlr}).

\begin{lemma}[Chernoff's bound]
  \label{lemma:chernoff}
  Let $n \in \NN$, $0 < p < 1$, and $0 < \eps \leq  3/2$.
  If $X$ follows a binomial distribution with parameters $n$ and $p$, then
  \begin{equation}
    \Pr(|X - \Exp[X]| \geq \eps\Exp[X]) \leq 2\exp\left(-\frac{\eps^{2}}{3}\Exp[X]\right).
  \end{equation}
\end{lemma}

Let $(A_i)_{i \in [k]}$ be a collection of increasing events, and
define $X = \sum_{i \in [k]}\mathds{1}_{A_i}$. If these events are mutually
independent,
we can apply a Chernoff-type bound to estimate the lower tail of $X$. When the
events are `mostly' independent, 
we use the following Janson-type
bound (see, for instance, Theorem 2.14 of~\cite{jlr}). 

\begin{lemma}[Janson's inequality]\label{lemma:janson}
  Let $(A_i)_{i \in [k]}$ be a collection of increasing events.
	Let 
	\[
    Z =  \sum_{i\in [k]} \mathds{1}_{A_i} \qquad \text{and} \qquad \Delta = \sum_{i \in [k]} \sum_{i \sim j} \mathbb{P}(A_i \cap A_j),
  \]
	where the sum is over non-ordered pairs $\{i,j\}$ of dependent events,
  with $i \neq j$. Then 
	\[
    \mathbb{P}\left(Z \le \Exp[Z]-t \right)\le \exp \left(- \dfrac{t^2}{2(\Exp[Z]+\Delta)}\right).
  \]
\end{lemma}

A sequence of random variables $(X_i)_{i=0}^{k}$ is a \defi{martingale} if
$\Exp[X_{i} | X_0,\dots,X_{i-1}] = X_{i-1}$, for every $i \in [k]$. We say that
a martingale $(X_i)_{i=0}^{k}$ is \defi{$c$-Lipschitz} if~${|X_i - X_{i-1}| \leq c}$,
for all~${i \in [k]}$. We will use the following Azuma-type bound
for martingales (see, for instance, Theorem 2.25 of~\cite{jlr}). 
	
\begin{lemma}[Azuma's inequality]\label{lemma:Azuma}
  Let $c>0$ and let $(Y_i)_{i=0}^k$ be a $c$-Lipschitz martingale.
  Then, for any $\lambda>0$, we have
	\[
    \P(|Y_k-Y_{0}|\geq \lambda)\leq 2\exp\left(-\frac{\lambda^2}{2kc^2}\right).
  \]
\end{lemma}

\subsection{On multigraphs}
\label{sec:graphs}

A \defi{multigraph} is a graph that allows multiple edges between the same pair of vertices.
The \defi{multiplicity} of a multigraph $G$, denoted by $\mu(G)$, is the maximum number
of edges joining the same pair of vertices.
Given an integer $k$, a \defi{$k$-edge colouring} of a multigraph $G$ is
a map $c:E(G) \rightarrow [k]$ with $c(e) \neq c(f)$ for any adjacent edges $e,f \in E(G)$.
The \defi{chromatic index} of a multigraph $G$, denoted by $\chi'(G)$, is the smallest integer $k$
for which a $k$-edge colouring of $G$ exists.
We will use Vizing's Theorem~\cite{vizing}.

\begin{thm}[Vizing's Theorem]
    \label{thm:vizing}
    For every loopless multigraph $G$, we have 
    \[
        \chi'(G) \leq \Delta(G) + \mu(G).
    \]
\end{thm}

\subsection{On digraphs}
\label{ssec:notation}

For a digraph $D$, we denote by $V(D)$ its vertex set 
and by $A(D)$ its arc set. If $(u,v) \in A(D)$,
we say that $v$ is an \defi{out-neighbour} of $u$, and~$u$
is an \defi{in-neighbour} of $v$. For ${v \in V(D)}$,
the set~$N^{+}(v)$ of out-neighbours of $v$ is called its \defi{out-neighbourhood}. The
\defi{out-degree} of~$v$ is $\deg^{+}(v) = |N^{+}(v)|$.
The \defi{in-neighbourhood}  $N^{-}(v)$ and the \defi{in-degree} $\deg^{-}(v)$ of $v$ are
defined analogously. The \defi{minimum semidegree} $\semidegree(D)$ of $D$ 
is the minimum over all the in- and all the out-degrees of the vertices. A \defi{sink} in a digraph is a vertex with out-degree $0$, while
a \defi{source} is a vertex with in-degree~$0$. 
We will use the symbols~$+$ and~$-$ to refer to the out- and in-related notions, respectively.
If~${X \subseteq V(D)}$, then for every vertex $v \in V(D)$ and $\diamond \in \{+,-\}$,
we set $N^{\diamond}(v,X) = N^{\diamond}(v)\cap X$.
We say  $D$ is~\defi{$d$-regular} if~$\deg^{\diamond}(v) = d$, for
all~${v \in V(D)}$ and all~${\diamond \in \{+,-\}}$.

The \defi{underlying multigraph}
of a digraph is a graph obtained by replacing each arc by an edge.
We write $\Delta(G)$ for the maximum degree of its underlying multigraph. 
An \defi{oriented} graph is a digraph that has no pair of vertices connected by arcs
in both directions. Note that the underlying multigraph of an oriented graph has multiplicity one, i.e., it is a graph.
In this case we simply refer to it as the \defi{underlying graph}.
An oriented graph is \defi{antidirected} if each vertex is either a source or a sink.

An \defi{oriented path} of length $k$ in $D$
is a sequence of distinct vertices~$(v_1,\dots,v_k)$ such that for
every $i \in [k-1]$ 
at least one of the arcs~$(v_i,v_{i+1})$ or~$(v_{i+1},v_i)$ exists in $D$.
Two oriented paths in $D$ have the \defi{same orientation} if the pattern
of arc directions along the paths is the same, up to reversing the order
of the vertices. 
A \defi{directed path} of length~$k$ in~$D$
is an oriented path where all arcs follow the same direction.
A \defi{directed cycle} of length $k$ is a sequence
of distinct vertices $(v_1,\dots,v_k)$ such that $(v_i,v_{i+1}) \in A(D)$, and
also~${(v_k,v_1) \in A(D)}$.
The \emph{$k$th power of a directed cycle}~$C$, denoted by $C^{k}$,
is the digraph obtained by adding
an arc from $v_i$ to $v_j$ whenever~$v_j$ is one of the next
$k$ vertices following~$v_i$ along~$C$.
An \defi{oriented Hamiltonian cycle} in a digraph $D$ on $n$ vertices is a sequence of 
vertices~${(v_1,\dots,v_n, v_{n+1}=v_1)}$ such that for every~${i \in [n]}$, either $(v_i, v_{i+1})$
or $(v_{i+1}, v_i)$ is an arc of $D$. A \defi{directed Hamilton cycle} in $D$
is an oriented Hamilton cycle where $(v_i,v_{i+1})$ is an arc of $D$,
for all $i \in [n]$.
An \defi{antidirected Hamilton cycle} is an oriented cycle where
the arcs alternate directions.

\subsection{On oriented trees}
\label{ssec:trees}
As usual, if $(T,r)$ is a rooted tree and  $uv \in E(T)$, with $u$ being closer to $r$
than $v$ is, then we say that~$v$ is a \defi{child} of $u$, and $u$ is
the \defi{parent} of~$v$.  
We work with the standard characterization of trees by the number of leaves or
bare paths. A \defi{bare path} in a graph is a path with the property
that each of its inner vertices has degree two. 
\begin{lemma}[{\cite[Lemma 2.1]{leafy-bare}}]
  \label{lemma:krivelevich}
  Let $k,n>2$ be integers. Every $T$  tree on $n$ vertices has at least $n/4k$
  leaves or a collection of $n/4k$ vertex-disjoint bare paths of length $k$.
\end{lemma}

For the proof of our main theorem, we need to
refine  this classification, taking into account
the orientations.
The following observation is crucial:
if a bare path in an oriented tree is not directed, then  at least one of its inner
vertices is a source or a sink. Any such vertex is called a \defi{switch}, and we say
that two switches are of the \defi{same type} if both are sources or both are sinks.
We will see next that in every oriented tree with many bare paths, many of these paths share
the same orientation. Further, either there are many directed paths or many switches. 
We also need to split our tree into two linear sized subtrees so that both contain the same structure.

\begin{lemma}
  \label{lemma:split_tree}
  Let $\Delta, n \geq 2$ and $k\geq 4$ be integers and set $t=n/(\Delta+1)^2$.
  Every oriented
  tree $T$ on $n$ vertices with $\Delta(T) \leq \Delta$ has an arc $e \in A(T)$ such
  that each of the two components $T_A$, $T_B$ of $T-e$   has at least $t$ vertices  and furthermore, one of the following three items holds for both $X\in\{A,B\}$:
  \begin{enumerate}[(a)]
	  \item  $T_X$ contains at least $\frac{t}{4k}$ leaves of $T$. 
	  \item   $T_X$ contains at least $\frac{t}{8k}$ vertex-disjoint directed bare paths of
      length $k$ of $T$.  
	  \item  $T_X$ contains at least $\frac{t}{16k}$ switches of the same type, and at least~$\frac{3t}{512}$ vertex-disjoint oriented bare paths of length 4 of $T$,
          all with the same orientation. 
          Moreover, the paths are vertex-disjoint
          from all switches and their neighbours.
\end{enumerate}
\end{lemma}

For the proof of this lemma we need the following result by Alon, Krivelevich and Sudakov~\cite{treesplitting}.

\begin{prop}[{\cite[Proposition 4.2]{treesplitting}}]
  \label{prop:two_trees}
  Let $\Delta \geq 1$ be an integer and $m \in \mathbb{R}$. Let $T$ be a tree on at
  least $m+1$ vertices with $\Delta(T) \leq \Delta$. Then there
  exists $e\in E(T)$ such that at least one of the trees obtained by deleting $e$ has
  at least $m$ and at most $\Delta m$ vertices.
\end{prop}

\begin{proof}[Proof of Lemma~\ref{lemma:split_tree}]
  We apply~Proposition~\ref{prop:two_trees} to the underlying graph of $T$
  with ${m=n/(\Delta+1)}$
  to obtain an  edge~${e_1 \in A(T)}$ that
  partitions $T$ into two subtrees~$T'$ and~$T''$ such that $m\leq |T'| \leq \Delta m$.
  Observe that, by the choice of $m$, we also have   $m=n-\Delta m \leq |T''| \leq n-m = \Delta m$. 
  We then apply Proposition~\ref{prop:two_trees} to  the underlying graph of~$T'$ with~${m'=|T'|/(\Delta+1)}$ to get an
  edge $e_2 \in A(T')$ that partitions $T'$ into $T_1$ and $T_2$ satisfying
  \begin{equation}
    \label{eq:bound_t}
    t = \frac{n}{(\Delta+1)^2} \leq m' \leq |T_1|\leq \Delta m' \leq \frac{\Delta^2n}{(\Delta+1)^2}.
  \end{equation}
  By the choice of $m'$, the tree $T_2$ also satisfies the bounds given in~\eqref{eq:bound_t}.
  Similarly, we can find an edge $e_3$ to split $T''$ into $T_3$ and $T_4$ satisfying the same bounds as in~\eqref{eq:bound_t}.
  We may assume that, for each $i \in [3]$, the  edge $e_i$ connects $T_i$ with $T_{i+1}$.

  First, suppose that each tree~$T_i$, for $i\in [4]$, has a collection as described in
  either \textit{(a)}, \textit{(b)} or \textit{(c)}. By the pigeonhole principle,
  at least two of the four trees, say $T_{i_1}$ and $T_{i_2}$, must have a
  collection of the same type. Let $e_i$ be an   edge separating $T_{i_1}$ from $T_{i_2}$.
  Then taking $T^{*}_1 = T[\bigcup_{j\le i} V(T_j)]$ and $T^{*}_2 =T[\bigcup_{j> i} V(T_j)]$,
  and observing that each of them has at least $t$ vertices, by~\eqref{eq:bound_t}, we get the result.

  To finish the proof  we only need to show that
  each tree $T_i$, for~${i \in [4]}$, has a collection as   in
  \textit{(a), (b),} or \textit{(c)}. Fix 
  ${i \in [4]}$.
  By Lemma \ref{lemma:krivelevich}, if
  $T_i$ does not satisfies~\textit{(a)}, then it contains~$t/4k$ vertex-disjoint bare paths of length $k$, which we assume from now on. If $T_i$ does not satisfy $\textit{(b)}$, then at least 
 $t/8k$ of the bare paths contain a switch. Then there is a set of $t/16k$ of the bare paths, all of which contain a switch of the same type. 
 From the other $3t/16k$ paths, we 
  select a collection of vertex-disjoint bare paths of length~$4$, all sharing the
  same orientation. 
  This collection contains at least $\floor{\frac{k}{4}}\frac{3t}{16k}\frac{1}{2^2} \geq \frac{3t}{512}$ 
  paths. 
  So~$T_i$ satisfies \textit{(c)}.
\end{proof}

\section{The Diregularity Lemma}
\label{ssec:diregularity}

Let~$D$ be a directed graph, let~$\eps > 0$, let~${\diamond \in \{+,-\}}$ and let~$X,Y\subseteq V$ be   disjoint
and nonempty. The \defi{out-density}~$d^{+}(X,Y)$ and \defi{in-density}~$d^{-}(X,Y)$ of the pair~$(X,Y)$ are defined by
\[
  d^{\diamond}(X,Y) = \frac{e^{\diamond}(X,Y)}{|X||Y|},
\]
where $e^{\diamond}(X,Y)$ counts the edges between $X$ and $Y$ in the direction indicated by~$\diamond$.
A set ${X' \subseteq X}$
is~\defi{$\eps$\nobreakdash-significant} if ${|X'| \geq \eps|X|}$.
The pair~$(X,Y)$
is~\defi{$(\eps,\diamond)$-regular}
if ${|d^{\diamond}(X,Y) - d^{\diamond}(X',Y')| \leq \eps}$ for
all~{$\eps$\nobreakdash-significant} subsets~${X' \subseteq X}$
and~${Y' \subseteq Y}$. If furthermore, for some $d \geq 0$, we
have~${\deg^{\diamond}(x,Y) \geq (d-\eps)|Y|}$, for all $x \in X$, we say
that the
pair~$(X,Y)$ is \defi{$(\eps,\diamond,d)$-super-regular}. 

The Regularity Lemma of Szemerédi~\cite{regularity} states that
every large graph can be partitioned into a bounded number of
sets, most of which are pairwise $\eps$-regular. We will use the
following version of this lemma for digraphs, obtained by Alon
and Shapira~\cite{alon}.

\begin{lemma}[Degree form of the Diregularity Lemma~\cite{alon}]
  \label{lemma:diregularity}
  For every $0<\eps < 1$ and $m_0 \in \NN$, there are integers $M_0$ and $n_0$
  such that the following holds for all $n \geq n_0$ and $0<d \leq 1$.
  If $D$ is a digraph on $n$ vertices, then there is a partition $\{V_0,\dots,V_k\}$ of $V(D)$, and
  a spanning subgraph~$D^{*}$ of~$D$ such that:
  \begin{enumerate}[(a)]
    \item $m_0 \leq k \leq M_0$,
    \item $|V_0| \leq \eps n$ and $|V_1|=\dots=|V_k|$,
    \item $\deg_{D^{*}}^{\diamond}(v) > \deg_{D}^{\diamond}(v) - (d+\eps)n$,
      for all $v \in V$ and $\diamond \in \{+,-\}$,
    \item $V_i$ is an independent set in $D^{*}$, for all $i \in [k]$, and
    \item for all $1 \leq i,j \leq k$, the ordered pair $(V_i,V_j)$ is
      $(\eps,+)$-regular in $D^{*}$ with density either $0$ or at least $d$.
  \end{enumerate}
\end{lemma}

We refer to the sets $V_1,\dots,V_k$ as \defi{clusters}, and call~$V_0$ and its vertices as \defi{exceptional}. The
spanning subgraph~${D^{*} \subseteq D}$ is called   a \defi{regularised digraph} of~$D$.
The \defi{$(\eps,d)$-reduced digraph} of $D$ with respect to $D^{*}$
and~$\{V_1,\dots,V_k\}$
is the digraph on the vertex set~$\{V_1,\dots,V_k\}$ that has an arc from~$V_i$ to~$V_j$ if and only
if~$(V_i,V_j)$ is~$(\eps,+)$-regular  with density at least~$d$.

It is known that the reduced digraph inherits the minimum semidegree
of~$D^{*}$ (scaled down to its order). Moreover, Kühn, Osthus, and Treglown~\cite{hamiltonian-degree}
showed that the reduced digraph also inherits the robust out-expansion property.

\begin{lemma}[{\cite[Lemmas 14 and 15]{hamiltonian-degree}}]
  \label{lemma:reduced}
  Let $1/n \ll \eps \ll d \ll \nu, \tau, \gamma < 1.$
  Let $D$ be a robust $(\nu,\tau)$-out-expander on $n$ vertices with $\semidegree(D) \geq \gamma n$.
  Let $D^{*}$ and~${\{V_0,\dots,V_k\}}$ be a regularised digraph
  of $D$ and a partition of~$V(D)$, respectively, obtained by applying
  Lemma~\ref{lemma:diregularity} with parameters~${\eps, d}$ and~$m_0$.
  Let~$R^{*}$ be the $(\eps,d)$-reduced digraph of $D$ with respect to $D^{*}$ and
  $\{V_1,\dots,V_k\}$. Then there exists a spanning oriented subgraph $R \subseteq R^{*}$ such that
  $\semidegree(R) \geq \gamma |R|/8$ and~$R$ is a robust
  $(\nu/24,2\tau)$-out-expander.
\end{lemma}

In the proof of Theorem~\ref{thm:main}, we will use the Blow-up Lemma.
We will use the following version, due to Csaba~\cite{csaba}.

\begin{lemma}[Blow-up Lemma]
  \label{lemma:blow_up}
  For all integers $\Delta, K_1, K_2, K_3$ and every positive constant~$c$
  there exists an integer $N$ such that whenever $\eps, \eps', \delta, d$
  are positive constants with
  \[
    0 < \eps \ll \eps' \ll \delta \ll d \ll 1/\Delta, 1/K_1, 1/K_2, 1/K_3, c
  \]
  the following holds. Suppose that $G$ is a graph of order $n \geq N$
  and $\{V_0,\dots,V_k\}$ is a partition of~$V(G)$ such that the pair
  $(V_i,V_j)$ is $\eps$-regular with density either $0$ or at least $d$
  for all~${1 \leq i < j \leq k}$. Let $H$ be a graph on $n$ vertices with
  $\Delta(H) \leq \Delta$ and let $\{L_0, L_1,\dots, L_k\}$ be
  a partition of~$V(H)$ with~${|L_i| = |V_i| = m}$, for every $i \in [k]$.
  Furthermore, suppose that there exists a
  bijection~${\psi:L_0 \rightarrow V_0}$ and a set $I \subseteq V(H)$
  of vertices at distance at least $4$ from each other such that the following 
  conditions hold:
  \begin{enumerate}
    \item[(C1)] $|L_0| = |V_0| \leq K_1 dn$.
    \item[(C2)] $L_0 \subseteq I$.
    \item[(C3)] $L_i$ is independent for every $i=1,\dots,k$.
    \item[(C4)] $|N_H(L_0) \cap L_i| \leq K_2 dm$, for every $i \in [k]$.
    \item[(C5)] For each $i \in [k]$, there exists $D_i \subseteq I \cap L_i$
                with $|D_i|=\delta m$ and such that
                for $D = \bigcup_{i=1}^{k}D_i$ and all $1 \leq i < j \leq k$,
                $\big| |N_H(D) \cap L_i| - |N_H(D) \cap L_j| \big| < \eps m$.
    \item[(C6)] If $xy \in E(H)$ and $x \in L_i, y \in L_j$ where $i,j \neq 0$
                then $(V_i,V_j)$ is $\eps$-regular with density $d \pm \eps$.
    \item[(C7)] If $xy \in E(H)$ and $x \in L_0, y \in L_j$, then
                $|N_{G}(\psi(x))\cap V_j| \geq cm$.
    \item[(C8)] For each $i \in [k]$, given any $E_i \subseteq V_i$ with
                $|E_i| \leq \eps'm$ there exists a
                set $F_i \subseteq (L_i \cap (I\setminus D))$ and a
                bijection $\psi_i:E_i \rightarrow F_i$ such that
                $|N_{G}(v) \cap V_j| \geq (d-\eps)m$ whenever
                $N_H(\psi_i(v))\cap L_j \neq \emptyset$, for all $v \in E_i$
                and all $j \in [k]$.
    \item[(C9)] Writing $F = \bigcup_{i=1}^{k}F_i$ we have that 
                $|N_H(F) \cap L_i| \leq K_3\eps'm$.
  \end{enumerate}
  Then $G$ contains a copy of $H$ such that the image of $L_i$ is $V_i$
  for all $i \in [k]$, and the image of each $x \in L_0$
  is $\psi(x) \in V_0$.
\end{lemma}

We remark that the way we state the Blow-up lemma is slightly different
from the version in~\cite{csaba}. First,
the additional properties of the copy of $H$ in $G$, namely that the image of~$L_i$
is~$V_i$, for all $i \in [k]$, and the
image of each $x \in L_0$ is $\psi(x) \in V_0$, 
do not appear in the statement but are explicitly included in the proof.
Moreover, our condition~(C6) allows the regular
pairs to have density close to~$d$ in terms of~$\eps$, while in the original statement
this density had to be exactly~$d$. This change does not affect the validity of the lemma,
as remarked by Csaba (see Remark~3 in~\cite{csaba}).
We also observe that, although the
Diregularity Lemma yields a regularised graph in which the regular pairs
have density at least~$d$, we can discard arcs
from pairs whose density exceeds~$d$ with appropriate probabilities. By Chernoff's bound,
the resulting density is close to $d$ in terms of $\eps$, allowing an application of the Blow-up lemma.
This gives the following well-known fact.

\begin{fact}
    \label{fact:d_close}
  Let $D$ be a digraph and let $(X,Y)$ be an $(\eps,+)$-regular pair
  in~$D$ with density at least $d$. Then there exists a spanning
  subgraph $D' \subset D$ such that $(X,Y)$ is an $(2\eps,+)$-regular
  pair in $D'$ with density $d \pm \eps$. If, moreover, the pair~$(X,Y)$
  is $(\eps,d)$-super-regular in $D$, then it is
  $(2\eps,d^2)$-super-regular in $D'$.
\end{fact}

Finally, 
to satisfy condition (C8) in the application of this lemma within our proof
of Theorem~\ref{thm:main}, 
it suffices to ensure that the arcs of a Hamilton cycle in the
reduced graph correspond to $(\eps,d)$-super-regular pairs
of cluster. It is well known that this can be done by removing a small
proportion of vertices from each cluster (see Proposition 10 in~\cite{kelly}).

\begin{fact}
  \label{fact:super_regular}
  Let $\Delta \in \NN$, and let $\eps \leq 1/(4\Delta)$. Also,
  let $D$ be a digraph, and let~$D^{*}$ and~${\{V_0,\dots,V_k\}}$ be the regularised digraph
  of $D$ and a partition of~$V(D)$, respectively, obtained by applying
  Lemma~\ref{lemma:diregularity} with parameters $\eps, d$ and $m_0$, for some $m_0 \in \NN$
  and~$d \leq 1$.
  Let $R$ be the $(\eps,d)$-reduced digraph of $D$ with respect to $D^{*}$ and $\{V_0,\dots,V_k\}$.
  If~$S$ is a subgraph of~$R$ with
  $\Delta(S) \leq \Delta$, then one can move exactly $\ceil{2\eps\Delta|V_i|}$ vertices 
  from each~${V_i \in V(S)}$ to~$V_0$ to obtain a new partition
  $\{V_0',\dots,V_k'\}$ of $V(D)$ such that
  \begin{itemize}
      \item $(V_i',V_j')$ is a $(2\eps,d - 3\eps\Delta)$-super-regular pair, for every $(V_i,V_j) \in A(S)$, and
      \item $(V_i',V_j')$ is $(2\eps, +)$-regular with density at least $d - 3\eps\Delta$, for every $(V_i,V_j) \in A(R)$.
  \end{itemize}
\end{fact}

We will also need the well-known Slicing Lemma (see, for instance, Fact 1.5 of~\cite{ks}).

\begin{fact}[Slicing Lemma]
  \label{fact:slicing}
  Let $D$ be a digraph, and let $(X,Y)$ be an $\eps$-regular pair
  in~$D$ with density at least $d$. For every $\alpha > \eps$, and
  $\alpha$-significant sets $X' \subseteq X$ and $Y' \subseteq Y$ the
  pair~${(X',Y')}$ is~$\eps'$-regular with density
  at least $d - \eps$, where $\eps' = \max\{\eps/\alpha, 2\eps\}$.
\end{fact}

\section{Robust expansion and skewed-traverses}
\label{sec:expansion}

Given a digraph~$D$ on $n$ vertices, a set~${X \subseteq V(D)}$, and a number $\nu > 0$,
the \defi{$\nu$-robust out-neighbourhood} of~$X$,
denoted by $\RN^{+}_{\nu,D}(X)$, is the set consisting of
all vertices in $D$ that receive at least $\nu n$ arcs from
vertices in $X$, i.e.
\[
    \RN^{+}_{\nu,D}(X) = \{ x \in V(D) : |N^{-}(x)\cap X| \geq \nu n\}.
\]
Recall that $D$ is a \defi{robust $(\nu,\tau)$-out-expander}, for some $\tau > 0$,
if $|\RN^{+}_{\nu,D}(X)| \geq |X| + \nu n$,
for all $X \subseteq V(D)$ with $\tau n < |X| < (1-\tau) n $. We define the
notions of \defi{robust in-neighbourhood} and \defi{robust in-expander} analogously.
Finally, we say that $D$ is a robust \defi{$(\nu,\tau)$-expander} if it is a robust $(\nu,\tau)$-out-expander
and also a robust $(\nu,\tau)$-in-expander. It is known
that robust out-expansion implies robust in-expansion (see Proposition 48 in~\cite{taylor}).
Moreover, between any pair of vertices in a robust out-expander, there exists a directed path of constant length.

\begin{lemma}[DeBiasio; Taylor (see Lemma 58 in~\cite{taylor})]
    \label{lemma:in_expansion_paths}
    Let $0 < \nu \leq \tau \ll \gamma < 1$. 
    Let $D$ be a robust~$(\nu,\tau)$-out-expander on~$n$ vertices with $\semidegree(D) \geq \gamma n$. Then 
    \begin{enumerate}
        \item $D$ is  a robust~$(\nu^3,2\tau)$-in-expander, and
        \item for every pair of distinct vertices $u,v \in V(D)$, there is a directed path
              of length~$\ceil{1/\nu}$ from~$u$ to~$v$.
    \end{enumerate}
\end{lemma}

By combining an application of the Diregularity Lemma with a straightforward use of Chernoff's bound (Lemma~\ref{lemma:chernoff}), Taylor showed (Lemma 60 in~\cite{taylor}) that every dense robust expander can be split
into two dense robust expanders of approximately equal size. The next lemma generalises this result
to partitions of arbitrary proportions. The proof is an immediate adaptation of Taylor's and introduces no new ideas,
so we omit the details.

\begin{lemma}
    \label{lemma:partition}
    Let $0 < 1/n \ll \nu' \ll \nu \leq \tau \ll \tau' \ll  \mu, \gamma < 1$.
    Let $D$ be a digraph on $n$ vertices with $\semidegree(D) \geq  \gamma n$
    and suppose that $D$ is a robust $(\nu,\tau)$-expander. Then 
    there is a partition~${\{A,B\}}$ of $V(D)$ such that
    \begin{enumerate}[(a)]
        \item $|A| = \ceil{\mu n}$,
        \item $|N^{\diamond}(v,A)| \geq \frac{\gamma}{2}|A|$ for all $v \in V(D)$ and  $\diamond \in \{+,-\}$, 
        \item $|N^{\diamond}(v,B)| \geq \frac{\gamma}{2}|B|$ for all $v \in V(D)$ and  $\diamond \in \{+,-\}$, and
        \item $D[A]$ and $D[B]$ are  robust $(\nu',\tau')$-expanders.
    \end{enumerate}
\end{lemma}

We close this section with the definition of skewed-traverses in digraphs.
Let~${C=(v_1,v_2,...,v_k)}$ induce a directed Hamilton cycle in a digraph~$D$. 
Given two vertices~$v_i$ and~$v_j$ in $D$, a \defi{$(v_i,v_j)$-skewed-traverse} is a collection of arcs of the
form~$\big( (v_iv_{i_1}),(v_{i_1-1}v_{i_2}),(v_{i_2-1},v_{i_3})\dots (v_{i_t-1}v_j)\big)$. See Figure~\ref{fig:skewed} for an example.
The \defi{length} of a $(v_i,v_j)$-skewed traverse 
is the number of arcs in the collection \footnote{We remark that Kelly defined the length of a skewed traverse as the number of arcs minus one, but for our purposes, it is more convenient to define it as the number of arcs.}. Taylor~\cite{taylor}
showed that in robust expanders, every pair of distinct vertices is connected by a skewed-traverse.
  
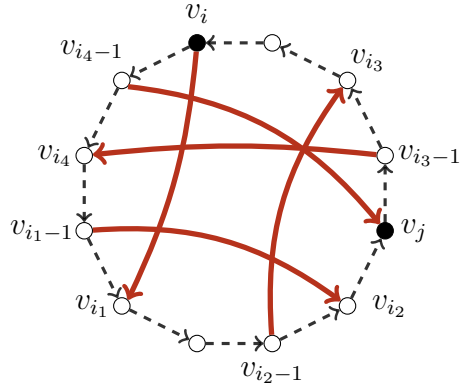
\begin{figure}[ht!]
    \centering
    \begin{tikzpicture}[scale=0.5]
		\node [style=vertex, fill=black] (0) at (4, 16) {};
        \node [above of=0, node distance = 11pt] {$v_i$};
		\node [style=vertex] (1) at (2, 15) {};
        \node [above left of=1, node distance = 16pt] {$v_{i_4-1}$};
		\node [style=vertex] (2) at (1, 13) {};
        \node [left of = 2, node distance = 11pt] {$v_{i_4}$};
		\node [style=vertex] (3) at (1, 11) {};
        \node [left of=3, node distance = 16pt] {$v_{i_1-1}$};
        \node [style=vertex] (4) at (2, 9) {};
        \node [left of=4, node distance = 11pt] {$v_{i_1}$};
        \node [style=vertex] (5) at (4, 8) {};
        \node [style=vertex] (6) at (6, 8) {};
        \node [below of = 6, node distance = 11pt] {$v_{i_2-1}$};
        \node [style=vertex] (7) at (8, 9) {};
        \node [right of = 7, node distance = 16pt] {$v_{i_2}$};
        \node [style=vertex, fill=black] (8) at (9, 11) {};
        \node [right of=8, node distance = 11pt] {$v_j$};
		\node [style=vertex] (9) at (9, 13) {};
        \node [right of = 9, node distance = 17pt] {$v_{i_3-1}$};
		\node [style=vertex] (10) at (8, 15) {};     
        \node [above right of = 10, node distance = 11pt] {$v_{i_3}$};
		\node [style=vertex] (11) at (6, 16) {};

		\draw [style=edge,dashed] (0) to (1);
		\draw [style=edge,dashed] (1) to (2);
		\draw [style=edge,dashed] (2) to (3);
		\draw [style=edge,dashed] (3) to (4);
		\draw [style=edge,dashed] (4) to (5);
		\draw [style=edge,dashed] (5) to (6);
		\draw [style=edge,dashed] (6) to (7);
        \draw [style=edge, dashed] (7) to (8);
        \draw [style=edge, dashed] (8) to (9);
        \draw [style=edge, dashed] (9) to (10);
        \draw [style=edge, dashed] (10) to (11);
        \draw [style=edge, dashed] (11) to (0);

        \draw [style=cobra, bend left=10] (0) to (4.north east);
        \draw [style=cobra, bend left=20] (3) to (7);
        \draw [style=cobra, bend left=20] (6) to (10);
        \draw [style=cobra, bend right=6] (9) to (2);
        \draw [style=cobra, bend left=22] (1.south east) to (8);

\end{tikzpicture}
    \caption{An example of a $(v_i,v_j)$-skewed-traverse of length 5.}
    \label{fig:skewed}
\end{figure}

  \begin{lemma}[\cite{taylor}]
    \label{lemma:skewed}
    Let $0 < \nu \leq \tau \ll \gamma < 1$.
    Let $D$ be a robust $(\nu,\tau)$-expander digraph on~$n$ vertices
    with $\semidegree(D) \geq \gamma n$.
    Let $C=(v_1,\dots,v_k)$ be a directed Hamilton cycle of $D$.
    For any distinct $u,v \in V(D)$,
    there exists a $(u,v)$-skewed-traverse of length at most~${\ceil{1/\nu}+ 1}$.
  \end{lemma}

\section{Random oriented walks}
  \label{sec:random_walks}
 
  Let $P=(p_i)_{i\in [m]}$ be an oriented path and let~$\diamond_i\in \{+,-\}$
  be such that~$p_{i+1} \in N^{\diamond_i}(p_i)$ for all $i\in[m-1]$. Given a
  digraph~$H$ on $k$ vertices, the~$P$-random walk starting at~${u\in V(H)}$ is a random sequence
  of vertices~$(X_i)_{i\in [m]}$ in~$V(H)$ such that~${X_1=u}$ and given
  an outcome~${X_1, \dots, X_\ell}$ with~${\ell < m}$, we choose~${X_{\ell+1}}$
  uniformly at random among the vertices of~$N^{\diamond_i}(X_\ell)$. We denote
  the probability distribution induced by this process by~${\mathbb{P}_{P,u}}$
  or~$\mathbb{P}_u$ when $P$ is clear from context.

We will show that our host graphs, namely robust expanders of a certain minimum semidegree, have the property that the distribution  of the last vertex~$X_m$ of a $P$-random walk tends to a uniform distribution. For this, the following property will be useful.

  \begin{defn}
    \label{def:cherry}
    Let~${\alpha > 0}$ and let~$D$ be a digraph on $k$ vertices. For any disjoint sets~${X,Y\subset V(D)}$
    and $\diamond \in \{+,-\}$, we define
    \[
      \mathcal{C}^\diamond(X,Y) = \big\{ (x,y,z): x \in X, y \in Y, z \in N^\diamond(x)\cap N^\diamond(y)  \big\}.
    \]
    \noindent We say~$D$
    has the~\defi{$\alpha$-cherry property} if for every partition $\{X,Y\}$ of the
    vertices~$V(D)$ and each 
    $\diamond \in \{+,-\}$,
    we have $|{\mathcal{C}^\diamond(X,Y)}|\ge \alpha |X||Y|k$.
    Moreover, we say that~$D$ has the \defi{cherry property} 
    if ${\mathcal{C}^\diamond(X,Y)}\neq \emptyset$ for all $\diamond \in \{+,-\}$.
  \end{defn}
  
  Let us explain  why the cherry property is a natural concept 
  when considering mixing of random walks. In the
  undirected world, it is well known that a random walk in a graph mixes well if the host
  graph is regular, connected and non-bipartite, and it is easy to check that a graph
  satisfies the (analogous undirected) cherry property if and only if it is
  connected and non-bipartite. In direct analogy, we show in Proposition \ref{proposition:mixing} that the mixing of the random walk in the directed setting is likely to happen in regular digraph with the cherry property. In Proposition \ref{proposition:regular-cherry-subgraph} we show that a linear minimum semidegree and robust expansion guarantees the existence of a regular subgraph with the cherry property. We now state these two results.

  \begin{prop}
    \label{proposition:regular-cherry-subgraph}
    Let $0 < 1/k \ll \nu \leq \tau \ll \gamma < 1$.
    Let $D$ be a digraph on~$k$ vertices   
    with~${\delta^0(D)\geq \gamma k}$ and suppose that~$D$ is a robust $(\nu,\tau)$-expander.
    Then there exists~$H\subset D$ that
    is~regular and has the cherry property.
  \end{prop}

  \begin{prop}
    \label{proposition:mixing}
    Let~$H$ be a regular digraph on $k$ vertices that
    satisfies the cherry property. Moreover, let~${P=(v_i)_{i\in [m]}}$ be an oriented
    path of length $m$. For every~${u\in V(H)}$, the $P$-random walk~${(X_i)_{i \in [m]}}$ starting at $u$ satisfies
    \[
      \max_{v\in V(H)} \left|\mathbb{P}_u(X_m=v) - \frac{1}{k} \right| \leq \exp\left(-\frac{m}{4k^5} \right).
    \]
  \end{prop}

  We will prove Proposition \ref{proposition:regular-cherry-subgraph} in Subsection \ref{ssec:subgraph} and Proposition \ref{proposition:mixing} in Subsection \ref{ssec:uniform}.

\subsection{Finding a regular subgraph with the cherry property}
\label{ssec:subgraph}

  We show the existence of the subgraph $H$ of $D$ from Proposition \ref{proposition:regular-cherry-subgraph} in three steps: first we show $D$ has the cherry property, then we show   a random subgraph of~$D$ has the cherry property with high probability, and finally, we carefully
  add some arcs  to make that subgraph regular. The obtained subgraph of $D$ has the cherry property since this property is monotone.

  \begin{lemma}
  \label{lemma:G-has-alpha-cherry}
    Let $0 < \nu \leq \tau \ll \gamma < 1$. Let~$D$
    be a robust $(\nu,\tau)$-expander
    with~${\semidegree(D) \geq \gamma k}$.
    Then $D$ has the~$2\nu^3$-cherry property.
  \end{lemma}

\begin{proof}
  Let $\{X, Y\}$ be a
  partition of $V(D)$, and let~${\diamond \in \{+,-\}}$. We can assume that~${|X|\ge |Y|}$.
  First assume that~$|Y| > \tau k$. Since $D$ is a robust $(\nu,\tau)$-expander, we have
  \[
    |\RN_\nu^{\diamond}(X)|+|\RN_\nu^{\diamond}(Y)| \geq |X| + |Y| + 2\nu k = (1+2\nu)k.
  \]
  \noindent and hence $|\RN_\nu^{\diamond}(X)\cap \RN_\nu^{\diamond}(Y)| \geq 2\nu k.$
  By considering each vertex in~${\RN_\nu^{\diamond}(X)\cap \RN_\nu^{\diamond}(Y)}$ and a pair of
  its in or out-neighbours (according to $\diamond$) in~$X$ and~$Y$, we conclude that
  $    |\mathcal{C}^{\diamond}(X,Y)|\geq 2\nu^{3}|X||Y|k$, which is as desired. 
  
  Now, assume that~${|Y| \leq \tau n \leq \gamma k/2}$. Let $\circ \in \{+,-\}\setminus\{ \diamond\}$, and choose an arbitrary~${y \in Y}$. For each $z \in N^{\circ}(y)$,
  we have
  \[
    |N^{\circ}(z,X)| \geq |N^{\circ}(z)| - |Y| \geq  \gamma k - |Y| \geq \frac{\gamma k}{2},
  \]
 and thus $|\mathcal{C}^{\circ}(X,Y)|\geq \frac{\gamma^2}{2}|Y| k^2 \geq 2\nu^3|X||Y|k$, as desired.
  \end{proof}

  Given a digraph $D$ and $0 \leq p \leq 1$, we denote by $D_p$ the random subgraph of $D$
  obtained by selecting each arc independently with probability $p$.
  We next show that if $D$ is a digraph satisfying the cherry property robustly,
  then $D_p$ also satisfies the cherry property for a wide range of values of $p$,
  via a straightforward application of Janson's inequality.
  \begin{lemma}
    \label{lemma:cherry_prop}
    For every $\alpha = \alpha(k)\in (0,1]$ and $p=p(k)\in (0,1)$
    satisfying~${\alpha^2 pk/\log k \rightarrow \infty}$ as $k \rightarrow \infty$,
    the following holds.
    If~$D$ is digraph on~$k$ vertices with the~$\alpha$-cherry property then with
    high probability~$D_p$ has the~$\alpha p^2/2$-cherry property.
  \end{lemma}

  \begin{proof}
    Throughout this proof, for a partition~$\{X, Y\}$ of $V(D)$ and $\diamond \in \{+,-\}$,
    we write~$\mathcal{C}^{\diamond}(X,Y)$ and~$\mathcal{C}_p^{\diamond}(X,Y)$ with the index signaling whether we are in~$D$
    or~$D_p$.
   First, fix a partition~$\{X, Y\}$ of $V(D)$ with $|X| \leq |Y|$, and
    let $\diamond \in \{+,-\}$. The variable $Z^{\diamond}_{X,Y} = |\mathcal{C}_p^{\diamond}(X,Y)|$ is then a sum of
    indicator functions for each cherry in~$\mathcal{C}^{\diamond}(X,Y)$.
    As~$D$ has the~$\alpha$-cherry property, we have
   $  \alpha |X||Y| k \leq |\mathcal{C}^{\diamond}(X,Y)| \leq |X||Y|k$,
    and consequently,
    \begin{equation} \label{eq:exp_Z}
      \alpha p^2|X||Y|k \leq \Exp[Z^{\diamond}_{X,Y}] \leq p^2|X||Y|k.
    \end{equation}

    We aim to apply Janson's inequality (Lemma~\ref{lemma:janson}) for estimate of the lower tail of $Z^\diamond_{X,Y}$.
    For this, define $\Delta$ as the expectation of pairs of distinct cherries    in~$\mathcal{C}_p^{\diamond}(X,Y)$ that have one arc in common. 
    The terms which contribute to~$\Delta$ come from
    sequences~$(x,y_1,y_2, z)$ with~$(x,y_i,z) \in \mathcal{C}_p^{\diamond}(X,Y)$
    for~$i \in [2]$, or $(x_1,x_2,y,z)$ with~${(x_i,y,z) \in \mathcal{C}_p^{\diamond}(X,Y)}$,
    for~$i \in [2]$.
    Therefore we have that
    \[
      \Delta \leq p^3|X|^2|Y|k + p^3|X||Y|^2k \leq p^3|X||Y|k^2,
    \]
    \noindent since $|X|+|Y|=k$. Note that the $\alpha$-cherry property can only hold for $\alpha \leq 1$ and therefore ${pk\geq\alpha^2 pk \gg \log k}$ for sufficiently large $k$, which entails that
    \[ \Exp[Z^{\diamond}_{X,Y}] \overset{\eqref{eq:exp_Z}}{\leq} p^2|X||Y|k \leq p^3|X||Y|k^2.
    \]
    So Janson's inequality with $t = \Exp[Z^{\diamond}_{X,Y}]/2$ implies that
    \[
      \mathbb{P}\left( Z^{\diamond}_{X,Y}\leq \frac{\Exp[Z^{\diamond}_{X,Y}]}{2}\right) \leq \exp \left(- \frac{(\alpha p^2|X||Y|k)^2}{8(\Exp[Z^{\diamond}_{X,Y}]+\Delta)} \right)\leq  \exp\left( - \frac{\alpha^2p|X||Y|}{16} \right).
    \]
    Taking the union bound over all~$X$ with at most~$k/2$ vertices, and $\diamond \in \{+,-\}$,
    we obtain
    \begin{align*}
      \mathbb{P}(D_p\text{ does not have the } \alpha p^2/2\text{-cherry property}) &\leq \sum_{\diamond \in \{+,-\}}\sum_{\substack{X \cup Y = V(D)\\|X|\leq\floor{k/2}}}\mathbb{P}\Big(Z^{\diamond}_{X,Y} \leq \frac{\Exp[Z^{\diamond}_{X,Y}]}{2}\Big)\\
      & \leq 2\sum_{i=1}^{\lfloor k/2 \rfloor}\binom{k}{i}\exp\left( - \frac{\alpha^2pi(k-i)}{16} \right)\\
      & \leq 2\sum_{i=1}^{\lfloor k/2 \rfloor} \exp\left(i\log k - \frac{\alpha^2pik}{32} \right)\\
      & \leq k\cdot \exp\left( - \frac{\alpha^2pk}{64} \right)
  \end{align*}
  \noindent which tends to $0$ since $\alpha^2 pk/ \log k \rightarrow \infty$ as $k \rightarrow \infty$.  
  \end{proof}

Now we prove that every subgraph $F$ of a robust $(\nu,\tau)$-out-expander digraph $D$
with minimum semidegree at least $\gamma k$ is contained in a spanning regular subgraph
$H\subset D$, as long as the maximum degree of~$F$ is not too large.
Our proof strategy follows an approach previously used by Mycroft and Naia in~\cite{tassio}, which consists of 
decomposing $F$ into matchings and iteratively complete the matchings into disjoint $1$-regular digraphs.

  \begin{lemma}
    \label{lemma:regular}
    Let $1/k \ll \xi \ll \nu \leq \tau \ll \gamma < 1$.
    Let  $D$ be a robust $(\nu, \tau)$-out-expander on~$k$ vertices
    with~${\semidegree(D) \geq \gamma k}$.
    For every ~$F \subseteq D$ with~$\Delta(F) \leq \xi k$, there
    exists a spanning regular digraph~${H \subseteq D}$ that contains $F$.
  \end{lemma}

  \begin{proof}
    Let~${\ell = \ceil{1/\nu}}$.
    Vizing's Theorem (Theorem~\ref{thm:vizing}), applied to the underlying multigraph
    of $F$ (which has multiplicity at most 2), yields a partition of~$A(F)$
    into~$d'$ disjoint matchings with~${d' \leq \Delta(F) + 2 \leq 2\Delta(F)}$.
    We further partition each of these matchings into at most~${\lceil 4\ell/\nu\rceil}$
    matchings to obtain a family of matchings~${M_1,\dots, M_d}$ satisfying
    \begin{enumerate}[(i)]
    		\item $|M_i|\leq \frac{\nu k}{8\ell}$, and
    		\item $d\leq 2\Delta(F)\ceil{\frac{4\ell}{\nu}} \leq \frac{10 \ell \xi k}{\nu} \leq\frac{\nu k}{8}$.
    \end{enumerate}

    We will construct a family~${H_1,\dots,H_d}$ of arc-disjoint $1$-regular 
    spanning subgraphs of~$D$ with~$M_i\subseteq H_i$, for each $i \in [d]$. The union of the $H_i$ will be the subgraph $H$ of the lemma.
    Suppose we already defined $H_1,\dots, H_{i-1}$ for some~${i \in [d]}$,
    and set $D_i = D \setminus \bigcup_{j < i}A(H_j)$.
    Let $a_1b_1,\dots,a_m b_m$ be an enumeration of $M_i$ such
    that~$a_j b_j \in A(D)$, for all $j \in [m]$.
    
    We claim $D_i$ contains a set of  pairwise vertex-disjoint directed paths~${P_1,\dots,P_m}$
    of length~$\ell$ such that~$P_j$ starts in~$b_j$ and ends
    in~$a_{j+1}$. 
    Indeed, we can construct these paths inductively by repeatedly applications of Lemma~\ref{lemma:in_expansion_paths}.
    To see this, note that for each $j \in [m]$, the digraph $D_i -\bigcup_{j' < j}V(P_{j'})$ has
    minimum semidegree at
    least~$\gamma k  - d - \ell m \geq \gamma k/4$, and is a robust $(\nu/2,2\tau)$-expander. This holds
    because we remove at most $d+\ell m \leq \nu k/8$ vertices from $D_i$ to obtain $D_i - \bigcup_{j' < j}V(P_{j'})$,
    and such deletion reduces the robust
    neighbourhood of any set by at most $\nu k/8$ vertices (see Proposition 35 in~\cite{taylor}).
    Then~${C_i = M_i \cup \bigcup_{j \in [\ell]}P_j}$ is  a directed cycle in $D_i$.
    Moreover, since~${k - \ell m \geq (1-\nu/8)k}$, Theorem~\ref{thm:hamilton} yields a
    directed cycle in $D_i$ that is Hamiltonian in~$D_i- C_i$. So ${H_i = C_i \cup C'_i}$ is as desired for the lemma.
  \end{proof}
We are now ready to prove the main result of this subsection.

\begin{proof}[Proof of Proposition~\ref{proposition:regular-cherry-subgraph}]
  Let $\xi > 0$ such that
  $\frac{1}{k} \ll \xi \ll \nu \leq \tau \ll \gamma < 1.$
  Consider $D_p$ with $p=\xi/2$. Note that $\nu^6 p k / \log k \rightarrow \infty$ as $k \rightarrow \infty$.
  Thus, by a straightforward application of Chernoff's bound (Lemma~\ref{lemma:chernoff}),
  it follows that w.h.p.~${\Delta(D_p)\leq 2pk = \xi k}$.
  Moreover, the combination of
  Lemmas \ref{lemma:G-has-alpha-cherry} and \ref{lemma:cherry_prop} implies that
  with high probability $D_p$ has the cherry property. Therefore there exists an outcome~$F$ of
  $D_p$ that satisfies both properties. Finally, Lemma~\ref{lemma:regular}
  applied to~$F$ yields the desired regular subgraph $H\subset D$ that
  contains~$F$. 
\end{proof}

\subsection{From the random directed walk to the uniform distribution}
\label{ssec:uniform}
In this subsection we prove Proposition \ref{proposition:mixing}, which measures
how close a random directed walk gets to the uniform distribution. For the sake of explanation, consider an oriented path
$P=(p_i)_{i\in [m]}$, a vertex $u$, and the $P$-random walk starting at $u$. Moreover let $f(v)=\mathbb{P}_u(X_i)$. In the next result we show that, in
a digraph with the cherry property, the maximum difference of $f$ among a pair of
vertices is to some extent achieved by a pair of vertices with a common
out(in)-neighbor. So if $X_i$ gets to that common neighbor, then the probabilities of those said vertices being reached at step $i+1$ will get closer to each other.

An interesting aspect of regular digraphs is that if you choose a
vertex $u$ uniformly at random and a vertex $v$ uniformly at random in $N^+(u)$,
then the distribution of $v$ itself is uniform. This simple fact greatly simplifies the analysis of the proof of Proposition \ref{proposition:mixing} because it gives us a clean way of describing how close $X_i$ is to the uniform distribution in terms of how close $X_{i-1}$ is. We couple that intuition with an analysis of the $L_2$-distance of~$X_i$ to the uniform distribution and a convenient Cauchy-Schwarz type  argument, where the cherry property appears.
We remark that a similar proof strategy was used by Mycroft and Naia~\cite{tassio}.

\begin{proof}[Proof of Proposition \ref{proposition:mixing}]
  Let $d \in \NN$ such that $H$ is $d$-regular.
  Let $P$ be given as in the proposition, and 
  let $\diamond_{i}, \circ_i \in \{+,-\}$ be such that $p_{i+1} \in N^{\diamond_i}(p_i)$ and $p_{i-1} \in N^{\circ_i}(p_i)$ for all suitable $i$.
  Let $(X_i)_{i \in [m]}$ be the given $P$-random walk starting at $u$.
  We wish to bound the distance from the distribution of $X_i$ to the
  uniform distribution, and will employ the squared euclidean
  distance to this aim. So for each $i \in [m]$, set
  \[
    m(X_i) = \sum_{w \in V(H)}\left(\mathbb{P}(X_i = w) - \frac{1}{k}\right)^2.
  \]
  
  Let $2 \leq i \leq m$. 
  For each~${v \in V(H)}$ let $f(v) = \mathbb{P}(X_{i-1} = v) - \frac{1}{k}$.
  As $H$ is $d$-regular, we have
  \begin{align*}
  \label{eq:function-m}
      m(X_i) &=\sum_{w\in V(H)}\left(\left(\sum_{v\in N^{\circ_i}(w)}\frac{\mathbb{P}(X_{i-1}=v)}{d}\right)-\frac{1}{k}\right)^2 \\
             &=\sum_{w\in V(H)}\left(\sum_{v\in N^{\circ_i}(w)}\frac{1}{d}f(v)\right)^2\\
             &\hspace*{-0.4cm}\stackrel{\text{Prop \ref{prop:cauchy}}}{=}
             \sum_{w\in V(H)}\left(\sum_{v\in N^{\circ_i}(w)}\frac{1}{d^2}\sum_{v \in N^{\circ_i}(w)}f(v)^2 - \frac{1}{2d^2}\sum_{u,v \in N^{\circ_i}(w)}(f(u) - f(v))^2\right)\\
             &=\frac{1}{d}\sum_{\substack{w\in V(H)\\v \in N^{\circ_i}(w)}}f(v)^2 - \frac{1}{2d^2}\sum_{\substack{w \in V(H)\\u,v \in N^{\circ_i}(w)}}(f(u) - f(v))^2.
  \end{align*}
  Observe that in the first term of the last line of the above equation, we are summing $f(v)^2$
  exactly $d$ times for each $v \in V(H)$, and divide the result by $d$, so this term equals $\sum_{v\in V(H)}f(v)^2=m(X_{i-1})$. Thus
  \begin{equation}
    \label{eq:mi}
    m(X_i) = m(X_{i-1}) - \frac{1}{2d^2}\sum_{\substack{w \in V(H)\\u,v \in N^{\circ_i}(w)}}(f(u)-f(v))^2.
  \end{equation}
  
  Now, let $M=\max_{u,v\in V(D)}|f(u)-f(v)|$. We will use the fact
  that $H$ has the cherry property to prove that the sum on the right hand side
  of~\eqref{eq:mi} is at least~${\left(\frac{M}{k}\right)^2}$.
  So, let $(v_i)_{i\in [k]}$ be a non-decreasing ordering of $V(H)$ 
  according to the values of $f$. Note
  that~${M = f(v_k)-f(v_1) = \sum_{i\in [k-1]} (f(v_{i+1})-f(v_i))}$.
  Hence, there exists $j\in [k-1]$ such that $f(v_{j+1})-f(v_j)\geq M/(k-1)$
  and consequently $f(v_\ell)-f(v_i)\geq M/(k-1)$ for every~${i,\ell \in [k-1]}$
  with $i\leq j < \ell$. Now let $X=\{v_1,..., v_j\}$ and $Y=V(H)\setminus X$.
  Since $H$ has the cherry property, there exist $u\in X$, $v\in Y$ and
  $w \in V(H)$ such that $u,v\in N^{\circ_i}(w)$. So, 
  \begin{equation}
    \label{eq:dif_fs}
    \sum_{\substack{w \in V(H)\\u,v \in N^{\circ_i}(w)}}(f(u) - f(v))^2 \geq \left(\frac{M}{k-1}\right)^2 \geq \left(\frac{M}{k}\right)^2.
    \end{equation}
 We claim that $M \geq \max_{v\in V(H)} |f(v)|$. Indeed, if $f$ is
  identical to~$0$ this holds trivially. Otherwise,
  let~${u \in V(H)}$ such that $|f(u)| = \max_{v \in V(H)}|f(v)|$.
  Given that~${\sum_{v \in V(H)}f(v) = 0}$, there exists~${u' \in V(H)}$
  with~$f(u')$ having opposite sign of~$f(u)$. 
  Hence,~${M \geq |f(u) - f(u')| \geq |f(u)| = \max_{v \in V(H)}|f(v)|}$.
  Furthermore, by an averaging argument, there
  exists $u \in V(H)$ such that $f(u)^2 \geq m(X_{i-1})/k$.
  Thus, $M \geq \max_{v \in V(H)} |f(v)| \geq \sqrt{m(X_{i-1})/k}$.
  Combining all of this with~\eqref{eq:mi} and~\eqref{eq:dif_fs}, we get
  \[
    m(X_i) \leq m(X_{i-1}) - \frac{M^2}{2d^2k^2} \leq m(X_{i-1}) - \frac{m(X_{i-1})}{2d^2k^3} \leq m(X_{i-1})\left(1-\frac{1}{2k^5}\right)\]
  \noindent and consequently
  \[
    m(X_i) \leq m(X_1)\left(1-\frac{1}{2k^5}\right)^{m-1}\leq m(X_1)\exp\left(-\frac{m}{4k^ 5}\right),
  \]
  where we used that $1-x \leq e^{-x}$, for all $x \in \mathbb{R}$,
  in the last inequality.
  Notice that $X_1$ is deterministically equal to $u$. Therefore,
  \[
    m(X_1) = \left(1-\frac{1}{k}\right)^2 + \frac{k-1}{k^2}<1.
  \]
  \noindent We finish the proof by deducing that
  \[\max_{v\in V(H)} \left|\mathbb{P}(X_m=v) - \frac{1}{k} \right| \leq m(X_m) \leq \exp\left(-\frac{m}{4k^ 5}\right).\]
\end{proof}

\section{Semi-random assignment}
\label{sec:semi-random}
Let $T$, $H$ be oriented graphs. An \defi{assignment} of~$T$ to~$H$
is a mapping~${\phi:V(T)\rightarrow V(H)}$ that preserves arcs,
i.e.~${xy \in A(T)}$ implies ~${\phi(x)\phi(y) \in A(H)}$. An assignment
is also called a homomorphism in other contexts.
Given a vertex~${v \in V(H)}$, a set~${S\subseteq V(T)}$ and an assignment~$\phi$
of~$T$ to~$H$, we are interested in the quantity $|\phi^{-1}(v)\cap S|$ when $T$ is a tree  and $H$ is a regular  digraph with
the cherry property. In the proof of our main result, $H$ is the reduced digraph and we use
the assignment to guide the actual embedding. 

A natural first approach would be to perform a randomised greedy embedding and use the mixing property of the random walks to show that each cluster of the reduced digraph is assigned roughly the same amount of vertices. This would indeed be sufficient for embedding almost spanning trees. However, since we are working with spanning trees,  we need to assign an exact amount of vertices to each cluster, and we also need to deal with the exceptional vertices of the regular partition. For this reason we will work with a semi-random assignment. This will have enough randomness to obtain a roughly balanced assignment and enough structure to correct the imbalance. 

We call  any ordering of the vertices of a rooted tree where each vertex is preceded by its parent a \defi{top-down ordering}. 

\begin{defn}[Semi-random assignment] \label{def:semi_random} 
 Let~$H$ be a digraph with a Hamilton cycle~$C$.
 Let $(T,r)$ be a rooted $n$-vertex tree with a top-down ordering $(u_1,u_2,\dots, u_n)$ of $V(T)$ starting at $u_1=r$. 
 Let $\cP$ be a collection of vertex-disjoint subtrees of $T$ that do not contain $r$.
 A \defi{$(H,C,\cP)$-semi-random assignment} of $T$ is an assignment~${\phi:V(T)\rightarrow V(H)}$ defined as follows.

	\begin{enumerate}[(1)]
		\item Choose $\phi(u_1)$ uniformly at random in $V(H)$.
		\item Suppose that $\phi(u_i)$ is defined for every $1\leq i< j\leq n$. Let $w$ be the parent of $u_j$ in $T$ and let $\diamond \in \{+,-\}$ be such that $u_j\in N_T^\diamond(w)$.
			\begin{enumerate}[(i)]
				 \item If $wu_j$ is in some $P\in \cP$, let $\phi(u_j)$ be the single vertex in $N^\diamond_C(\phi(w))$,
                 and
				 \item otherwise, choose $\phi(u_j)$ uniformly at random from $N_H^\diamond(\phi(w))$.
			\end{enumerate}
	\end{enumerate}
\end{defn}

In the next proposition we exhibit some properties of the semi-random assignment if~$H$ is regular and
has the cherry-property. Given a rooted oriented tree~$(T,r)$, and a collection~$\cP$ of vertex-disjoint
subtrees of $T$, we let \defi{$\Root(\cP)$} consist of one vertex from each~$P\in\cP$, namely the one that is 
closest to $r$. Moreover, we set $V(\cP)=\bigcup_{P\in\cP }V(P)$.

\begin{prop}
  \label{prop:sr_assignment}
 For every $\Delta, k,M \in \NN$,
 there are $\varepsilon>0$ and $n_0$ 
 such that the following holds
 for all~$n \geq n_0$. Let $H$ be a regular $k$-vertex 
 digraph on $k$ vertices with the cherry property and  a Hamiltonian cycle $C$. Let~$(T,r)$ be an $n$-vertex oriented rooted tree with $\Delta(T) \leq \Delta$ and let $\cP$ be a set of vertex-disjoint subtrees of~$T$, each on at most $M$ vertices. Then there is a subset $\cP'\subset \cP$ with $|\cP'|\geq |\cP| - n^{1/2}$ such that the following holds for any fixed  $S\subset (V(T)\setminus V(\cP'))\cup \Root(\cP')$ and $v\in V(H)$.
 Any $(H,C,\cP')$-semi-random assignment $\phi$ of $T$ satisfies
 \begin{enumerate}[(a)]
 	\item $|\phi^{-1}(v)\cap S| = \frac{|S|}{k} \pm n^{1-\varepsilon},$ and
  \item $|\phi^{-1}(v)\cap V(\cP')| = \frac{|V(\cP')|}{k}\pm n^{1-\varepsilon}, $ 
  \end{enumerate}
 with probability at least $1-\exp(-n^{\varepsilon})$.
\end{prop}

We will prove Proposition \ref{prop:sr_assignment} at the end of this subsection. 
In order to shed some light on its proof,  consider the case where $\cP=\emptyset$ and $S=V(T)$.
The random assignment of Proposition \ref{prop:sr_assignment} is then a randomised
breadth first search algorithm with repetitions allowed and the proof follows by
a martingale analysis. We know by Proposition~\ref{proposition:mixing} that a
random oriented walk mixes in a regular digraph with the cherry property.
Therefore, we expect the vertices of $V(T)$ to be well distributed as long
as~$T$ is much larger than the host graph~$H$. However, we have to split the
tree into smaller subtrees and embed them one at a time to guarantee that that
expectation actually happens with high probability. 

We define random variables that count how many vertices of each subtree are
assigned to a particular vertex of~$H$, and consider the  Martingale arising naturally from this.
Each subtree can change the balance of the assignment by at most its size, which is a
just strong enough Lipschitz condition for Azuma's
inequality (Lemma~\ref{lemma:Azuma}), provided that the subtrees have
sublinear size. The argument remains the same if we count how many vertices of a
fixed set~${S\subset V(T)}$ is assigned to a vertex of $H$. The case when $\cP\neq\emptyset$  is similar as we can adjust the argument to calculate the expectations in the deterministic steps. With this motivation, let us move to admissible tree partitions.

Given a rooted tree $(T,r)$, a  family of disjoint rooted
subtrees $\big\{(T_i,r_i)\big\}_{i=0}^s$ is said to be \defi{admissible} 
if 
\begin{enumerate}[a)]
    \item $r = r_0$,
    \item $V(T) = \bigcup_{i=0}^{s}V(T_i)$, and
    \item if an edge $vw$ meets both $V(T_i)$ and $V(T_j)$ for 
          distinct $i,j\in \{0\}\cup [s]$, then  $\{v,w\}\cap\{r_i,r_j\}\neq\emptyset$. Moreover, $i<j$ if and only if $r_i$ is closer to $r$ than $r_j$ is.
\end{enumerate}

\begin{lemma}
  \label{lemma:cutting}
  For all $K,\Delta\in\mathbb N$ and for every rooted tree $(T,r)$
  with~${\Delta(T) \leq \Delta}$, there exists an admissible
  family $\big\{(T_i,r_i)\big\}_{i=0}^s$
  of disjoint rooted subtrees of $T$ such that $v(T_0)\leq K$
  and~$K \leq v(T_i) \leq \Delta K$, for every $i \in [s]$.
\end{lemma}

\begin{proof}
 Let~$T(u)$ be the subtree of $T$ induced by all vertices whose path to $r$ contains $u$.
  We construct the family~$(T_i,r_i)_{i\in [s]}$ in at most~$v(T)/K$ steps.
  Let~$S_0 = T$.
  For every~${i \geq 1}$, choose a vertex~${r_i \in V(S_{i-1})}$ (if it exists)
  such that $|S_{i-1}(r_i)| \geq K$ and that $|S_{i-1}(v)| \leq K$ for every
  child $v$ of  $r_i$ in $S_{i-1}$, further,   set~$T_i = S_{i-1}(r_i)$, and $S_{i} = S_{i-1} \setminus S_{i-1}(r_i)$ and root~$T_i$ at $r_i$. If no such vertex exists, then set $r_i=r$ and $T_i=S_{i-1}$ and terminate the process. Note that each $T_i$ but the last has between $K$ and $\Delta K$ vertices, and the last tree has less than $K$ vertices.
  By construction, we have $V(T) = \bigcup_{i=0}^{s}V(T_i)$. We invert the ordering of the trees so that $r_0=r$ and $v(T_0)\leq K$. Then  properties $a)-c)$ hold.
\end{proof}

\begin{proof}[Proof of Proposition \ref{prop:sr_assignment}]
  Set
  \[
    \gamma = \frac{1}{1+9k^5\log\Delta}, \quad
    \delta=\frac{\gamma}{10}, \qand
    \varepsilon = \min\left\{\frac{1}{4+12k^5\log\Delta},\frac{\delta}{4}\right\}.
  \]  
  Also, let $L$ be the number of non-isomorphic rooted oriented trees on $M$ vertices.
  We choose $n_0$ such that~${1/n_0 \ll \gamma,\eps,k,\Delta,L}$. Now, let $n \geq n_0$.
  Given $H$, $C$, $T$ and $\mathcal P$, 
  we apply Lemma~\ref{lemma:cutting} to the underlying graph of the rooted tree given by $(T,r)$,
  with~${K=n^{1-5\delta}}$ and~$\Delta$, to obtain an admissible
  family $\big\{(T_i,r_i)\big\}_{i=0}^s$ of
  disjointed rooted trees with $|T_0|\leq K$ and~${K\leq |T_i|\leq \Delta K}$
  for $i\in [s]$.
  
  Order each $V(T_i)$ in a top-down fashion, and concatenate these orderings to a top-down ordering $(v_1, v_2,\dots, v_n)$ of $V(T)$ so that   $V(T_i)$ comes before $V(T_j)$
  if $i < j$ . Obtain $\cP'$  from $\cP$ by deleting each $P$ within distance at most
  \begin{equation} \label{eq:m}
    m=\Big\lceil \frac{\log n^{1-\gamma}}{2\log\Delta}\Big\rceil-2
  \end{equation}

  \noindent from any of the~$r_i$. Note that we lose at most as many trees from $\cP$ as vertices within
  that distance from the roots $r_i$. Therefore, 
  \[
    |\cP'|\geq |\cP| - (s+1)\sum_{i=0}^{m}\Delta^i \geq |\cP| - \dfrac{2n}{K}\cdot\Delta^{m+1}\geq |\cP|- n^{1/2}.
  \]
 
  \noindent Fix $v \in V(H)$ and $S \subset (V(T)\setminus V(\cP'))\cup \Root(\cP')$.
  Let~$\varphi$ be any $(H,C,\cP')$-semi-random
  assignment of $T$.
  We start by proving (a).
  For every $0 \leq i \leq s$,
  we denote by~$X_i$ the number of vertices of $S$ in $T_i$ that are assigned to $v$ by $\varphi$, that is~${X_i=|\phi^{-1}(v)\cap V(T_i)\cap S|}$.
  We also set $Y=\sum_{i=0}^s X_i$.
  For $i=0,1,\dots, s$ we define the $\sigma$-algebras $\mathcal{F}_i=\sigma(X_0,X_1,\dots, X_i)$ and the martingale sequence $(Y_i)_{i=-1}^s$ by
  setting $Y_{-1}= \Exp[Y]$ and $Y_i= \Exp[Y | \mathcal{F}_i]$
  for~${i=0,1,\dots, s}$. Let us determine the Lipschitz property of this
  martingale sequence. For a fixed $i$ and a vector $x=(x_0,x_1,\dots, x_i)$, with $0\leq x_i \leq |T_i|$, let
  $\mathcal{E}_j = \cap_{\ell=0}^j \{X_\ell=x_\ell\}$ for each~${j=0,1,\dots, i}$
  and let $\mathcal{E}_{-1}$  be the whole probability space. The random
  variable $Y_{i}-Y_{i-1}$ is $\mathcal{F}_i$-measurable and takes a
  vector~$x$ to
 \begin{equation}
		\label{eq:Liptchitz-condition1} 
    Y_i-Y_{i-1} = \big(x_i-\Exp[X_i|\mathcal{E}_{i-1}] \big) + \sum_{h=i+1}^s\big(\Exp[X_h|\mathcal{E}_i]-\Exp[X_h|\mathcal{E}_{i-1}] \big).
 \end{equation}
 
 Note that the first term must satisfy
 \begin{equation}
 	\label{eq:xi-term1}
 	\big|x_i-\Exp[X_i|\mathcal{E}_{i-1}]\big| \leq |T_i|\leq \Delta K,
 \end{equation}
 
  \noindent where the first inequality is true since $0 \leq x_i, X_i \leq |T_i|$. Also note that
   for $h=i+1,\dots,s$ the value of $X_h$ only depends
  on $\phi(p_h)$, where $p_h$ is the parent  of $r_h$.
  Therefore, for every~${j\in \{i-1,i\}}$ and~${h>j}$ we have that

  \begin{equation*}
	 \begin{split}
 	\Exp[X_h | \mathcal{E}_j] &=  \sum_{w\in V(T_h)\cap S}\sum_{u\in V(H)}\P\big(\phi(w)=v  |\mathcal{E}_j \cap \{\phi(p_h)=u\}\big) \cdot	\P(\phi(p_h)=u)\\
 	&= \sum_{w\in V(T_h)\cap S}\sum_{u\in V(H)}\P\big(\phi(w)=v | \{\phi(p_h)=u\}\big) \cdot \P(\phi(p_h)=u).
 	\end{split}
 \end{equation*}
 
  We split the above summation by considering vertices from $V(T_h)\cap S$ at
  distance at most $m$ from $p_h$, which amount to at
  most $\Delta^{m+1}= O(n^{1/2 + 5\delta})$ vertices, separately from the vertices at a larger
  distance. We will now show by induction that for each vertex~${w\in V(T_h)\cap S}$ at distance at least $m$ from $p_h$, the
  $(H,C,\cP')$-semi-random assignment~$\phi$, conditioned on $\phi(p_h)=u$, satisfies
  \begin{equation}
  	\label{eq:prob-success}
    \P\big(\phi(w)=v | \{\phi(p_h)=u\}\big) = \frac{1\pm n^{-\gamma}}{k}.
  \end{equation}
 \noindent The base case is  $w$ being at distance $m+1$ from $p_h$. By definition of $\cP'$, the path $P$ between~$p_h$ and $w$ has no vertices from $V(\cP')$, and therefore the semi-random assignment of this path induces just a $
 P$-random walk. Then \eqref{eq:prob-success} follows by~\eqref{eq:m} and Proposition \ref{proposition:mixing}. For vertices of greater distance we consider two cases. If the parent of $w$ is $w'$ and $w'w$ belongs to a tree in $\cP'$, then the next step is deterministic. Let $\diamond \in \{+,-\}$ be such that $w\in N^{\diamond}(w')$ and let $v'\in V(H)$ be the vertex such that $v\in N^{\diamond}_C(v')$. If we assign $w'$ to $v'$, then we are forced to assign $w$ to $v$. By induction, we have that 
  \[ \P\big(\phi(w)=v | \{\phi(p_h)=u\}\big) = \P\big(\phi(w')=v' | \{\phi(p_h)=u\}\big) =\frac{1\pm n^{-\gamma}}{k}. \]
 
 If $w'w$ does not belong to a tree in $\cP'$, the argument is similar. We apply induction to calculate the probability of assigning $w'$ to each of the neighbours of $v$, with the right orientation, and multiply by the chance of reaching $v$ in the next step. Since $H$ is regular, we obtain \eqref{eq:prob-success}. We conclude that 
 
  \begin{equation}
	\label{eq:expXk1}  
		 (1- n^{-\gamma})\frac{|V(T_h)\cap S|}{k}\leq\Exp[X_h | \mathcal{E}_j] \leq \Delta^{m+1} + (1+ n^{-\gamma})\frac{|V(T_h)\cap S|}{k}.
 \end{equation}
 Combining \eqref{eq:Liptchitz-condition1} with 
   \eqref{eq:xi-term1}
  and \eqref{eq:expXk1},  we arrive at 
   \begin{equation}\label{xxxxx}
    |Y_i-Y_{i-1}|\leq \Delta K + 2(s+1)\Delta^{m+1}+\sum_{h=i+1}^s \left(2n^{-\gamma} \cdot   \frac{|V(T_h)\cap S|}{k}\right).
  \end{equation} 
 
  For the first term of~\eqref{xxxxx}, recall that $\Delta K = \Delta n^{1-5\delta}$.
  For the second term of~\eqref{xxxxx}, note that $s\leq n^{5\delta}$ for
  every $\delta \leq 1/10$, since $|T_h|\geq n^{1-5\delta}$,
  for every $h\in [s]$. Therefore, we have
  \[
    2(s+1)\Delta^{m+1} \leq 4n^{5\delta}\Delta^{m+1} =O\big(n^{1/2} \big),
  \]

  \noindent 
  The last term of~\eqref{xxxxx} is at
  most $2n^{1-\gamma}=2n^{1-10\delta}$, since the subtrees $T_h$ form a partition of
  an $n$-vertex tree. Combining all these bounds we have that
  \[
    |Y_i-Y_{i-1}| \leq n^{1-4\delta},
  \]

  \noindent which means that the martingale sequence $(Y_i)_{i=-1}^s$
  is $n^{1-4\delta}$-Lipschitz. To finish the argument, we only have to find the
  value of $Y_{-1}=\Exp[Y]$. 
  Similarly as in \eqref{eq:expXk1},
  but without considering each subtree separately, we see that $v_1$ is
  assigned to $v$ with probability $1/k$ and all other vertices, except the
  at most $\Delta^{m'+1}$ at distance at most $m'=\ceil{\frac{\log n^{1-3\eps}}{\log\Delta}} - 1$
  from $v_1$, are assigned to $v$ with probability $(1\pm n^{-3\varepsilon})/k$.
  Therefore, 
  \begin{equation*}
		Y_{-1} = O(\Delta^{m'+1}) + (1\pm n^{-3\varepsilon}) \frac{|S|}{k}
           = O(n^{1-3\eps})+(1\pm n^{-3\varepsilon}) \frac{|S|}{k}
    	     = \frac{|S|}{k} \pm n^{1-3\varepsilon}.
  \end{equation*}

  So, Azuma's inequality (Lemma \ref{lemma:Azuma}) implies that
  \[
    \P(|Y_s-Y_{-1}|\geq n^{1-\delta})\leq 2\exp\left(-\frac{n^{2-2\delta}}{2sn^{2-8\delta}}\right)\leq \exp\left(-\frac{n^\delta}{2}\right)\leq \exp(-n^{3\varepsilon}),
  \]
		
 \noindent Thus, with probability at least $1-\exp(-n^{3\varepsilon})$, we have 
  \[
    Y_s= \frac{|S|}{k} \pm (2n^{1-3\varepsilon}+n^{1-\delta}) = \frac{|S|}{k} \pm n^{1-2\varepsilon},
  \]
  Now we recall that $Y_s=Y_s(v)$ is a function of $v$. 
  By the union bound, the probability of $Y_s(v)=\frac{|S|}{k} \pm n^{1-2\varepsilon}$ is at least 
  $1-n\exp(-n^{3\varepsilon})\geq 1-\exp(-n^{2\varepsilon})$. This proves item (a),
  since $1-\exp(-n^{2\eps}) \geq 1-\exp(-n^\eps)$ and $n^{1-2\eps} \leq n^{1-\eps}$.

  We now prove item (b) of the proposition. We say that two rooted oriented trees $(P,r)$ and $(P',r')$ are isomorphic if there is a bijection $f:V(P)\rightarrow V(P')$ that preserves oriented edges and such that $f(r)=r'$. First suppose that $\cP'$ has only one isomorphism class and let $S = \Root(\cP')$. By the proof of item (a), with probability at least $1-\exp(-n^{2\varepsilon})$ we have
  \[|\phi^{-1}(v)\cap S| = \dfrac{|S|}{k} \pm n^{1-2\varepsilon}\]
  for every $v\in V(H)$. Since the assignment of each subtree to $V(C)$ is deterministic once the first vertex is assigned, 
  and since each subtree has at most $M$ vertices, we have
  \[
    |\phi^{-1}(v)\cap V(\cP')| = \dfrac{|V(\cP')|}{k} \pm Mn^{1-2\varepsilon}.
   \]
  \noindent Now, if $\cP'$ has more than one isomorphism class, we just sum the above equation over these classes.
  Since the number of classes is at most $L$, we obtain that
  \[
    |\phi^{-1}(v)\cap V(\cP')| = \dfrac{|V(\cP')|}{k} \pm Ln^{1-2\varepsilon},
  \]
  with probability at least $1-L\exp(-n^{2\varepsilon})\geq 1-\exp(-n^{\varepsilon})$. This finishes the proof of (b) since $Ln^{1-2\varepsilon}\leq n^{1-\varepsilon}$. 
\end{proof}

\section{Proof of Theorem~\ref{thm:main}}
\label{sec:proof_main}

From the premises of the theorem we have $0<1/n \ll \nu \leq \tau \ll \gamma, \Delta^{-1}$, which also satisfies
the hierarchies of Theorem~\ref{thm:hamilton}, Lemma~\ref{lemma:in_expansion_paths}, and Lemma~\ref{lemma:skewed}.
We insert constants~$\nu'$ and $\tau'$ as in Lemma~\ref{lemma:partition} such that
\[
    1/n \ll\nu' \ll \nu \leq \tau \ll \tau'\ll \gamma, \Delta^{-1}.
\]
We also insert certain $\eps_i$ and $M_i$, for $i\in[4]$. Namely, let $n_i, M_i$ be given by the
Diregularity Lemma (Lemma~\ref{lemma:diregularity}) for input $\eps_i$, $m_0=1/\eps_i$, for $i\in[4]$.
For ease of notation, we set
\begin{equation}
    \ell = 6\Big\lceil\frac{8}{\gamma}\Big\rceil M_4.
\end{equation}
Let $\beta$ and $n_5$ be given by Proposition~\ref{prop:sr_assignment} for input $\Delta$, $k=\eps_1^{-1}$, $M=\ell$.
We require $n\ge\max_{i\in[5]}n_i$. This gives the following  hierarchy of constants, where we also
insert~$\beta$ (from Proposition~\ref{prop:sr_assignment}), and further parameters $\delta_1, \delta_2, d, \eps^*_1,\eps^*_2$.
{\footnotesize
\begin{equation}
  \label{eq:parameters}
  \hskip-.2cm 0 < \frac{1}{n} \ll\beta \ll \frac{1}{M_1} \ll \eps_1 \ll \frac{1}{M_2} \ll \eps_2 \ll \eps^{*}_1 \ll \delta_1 \ll \frac{1}{M_3} \ll \eps_3 \ll \frac{1}{M_4} \ll \eps_4 \ll \eps^{*}_2 \ll \delta_2 \ll d 
  \ll \nu'. 
\end{equation}
}

\noindent Let $D$ be a robust $(\nu,\tau)$-out-expander on $n$ vertices with $\semidegree(D) \geq \gamma n$, and
let $T$ be an oriented tree on~$n$ vertices with $\Delta(T) \leq \Delta$.

\subsection{Preparing $T$}
\label{ssec:cases}
We set the following auxiliary constant:
\[
    \lambda = \frac{1}{4(\Delta+1)^2}. 
\]
Apply Lemma~\ref{lemma:split_tree} with parameters $\Delta, n$,
and $\ell$, to obtain an edge $e = (r_A,r_B)$ of $T$ such that removing
$e$ from $T$ yields two
trees~$T_A$ and~$T_B$, each of which contains
at least~${n/(\Delta+1)^2}$ vertices. 
We may assume that $|T_B| \geq |T_A|$.
Furthermore, both trees $T_A$ and $T_B$ have
either
\begin{enumerate}
  \item[(T1)] a set of at least $\lambda n/\ell$ leaves, or
  \item[(T2)] a set of at least $\lambda n/2\ell$ vertex-disjoint
        directed bare paths on $\ell$ vertices, or
  \item[(T3)] a set of at least $\lambda n/4\ell$
        switches and a set of at
        least~$3\lambda n/128$ 
        paths on $4$ vertices all with the same orientation.
        Moreover, all the paths in these two collections are vertex disjoint.  
\end{enumerate}

We say that the tree $T$ is \defi{leafy},~\defi{bare}, or \defi{switchy}
depending on whether $T_A$ and $T_B$ satisfy condition~(T1), (T2), or (T3),
respectively. Furthermore, we select a specific subset of the parameters defined
in~\eqref{eq:parameters} based on the type of $T$.
If $T$ is leafy or switchy, we set
\begin{equation*}
  \eps_A = \eps_1,\quad
  \eps_B = \eps_2,\quad
  \eps^{*} = \eps^{*}_1, \quad
  \delta = \delta_1,\quad 
  M_A = M_1,\qand
  M_B = M_2;
\end{equation*}
otherwise we set
\begin{equation*}
  \eps_A = \eps_3,\quad
  \eps_B = \eps_4,\quad
  \eps^{*} = \eps^{*}_2,\quad
  \delta = \delta_2,\quad
  M_A = M_3,\qand
  M_B = M_4.
\end{equation*}
Note that in both cases we have the following hierarchy of constants
\begin{equation}\label{hiera2}
  0 < \beta \ll \frac{1}{M_A} \ll \eps_A \ll \frac{1}{M_B} \ll \eps_B \ll \eps^{*} \ll \delta \ll d \ll \nu'\ll \gamma < 1,
\end{equation}
which also satisfies the hierarchies of Propositions~\ref{proposition:regular-cherry-subgraph} and~\ref{proposition:mixing} with $k$ taken as either $\eps_A^{-1}$ or $\eps_B^{-1}$.

\subsection{Preparing $D$}\label{ssec:preparing_D}
By Lemma~\ref{lemma:in_expansion_paths}, we have
that $D$ is a robust $(\nu^3,2\tau)$-expander.
We introduce an auxiliary constant $\mu' > 0$ such
that
\begin{equation} \label{eq:mu_p}
    (1-3\eps_A/4)(1-\eps_B)\mu' + \eps_B = |T_A|/n.
\end{equation}
By applying Lemma~\ref{lemma:partition} with $\mu'$ and $\gamma$, we
obtain a
partition~${\{A', B'\}}$ of $V(D)$ such that:
\begin{enumerate}[(D1')]
  \item $|A'| = \ceil{\mu' n}$,\label{G1}
  \item $|N^{\diamond}(v,A')| \geq \frac{\gamma}{2}|A'|$, for all $v \in V(G)$ and $\diamond \in \{+,-\}$, 
  \item $|N^{\diamond}(v,B')| \geq \frac{\gamma}{2}|B'|$, for all $v \in V(G)$ and $\diamond \in \{+,-\}$, and
  \item $D[A']$ and $D[B']$ are robust $(\nu',\tau')$-expanders.
\end{enumerate}

Apply the Diregularity Lemma (Lemma~\ref{lemma:diregularity}) to $D[B']$
with parameters $\eps = \eps_B/16$, $m_0 = m_B$,
and~${d_B = d + 15\eps_B/16}$
to obtain a partition of order~${k_B\geq m_B}$ of~$B'$, as well as a regularised digraph~$D^{*}_{B'}$ of~$D[B']$.
Let~$R_{B}^{*}$ be the corresponding~$(\eps_B/16,d_B)$-reduced graph. 
By Lemma~\ref{lemma:reduced},  there exists a spanning oriented subgraph $R_B \subseteq R_B^{*}$
with~$\semidegree(R_B) \geq \frac{\gamma}{8}|R_B|$ and $R_B$ is a robust~$(\nu'/24,2\tau')$-expander.
Thus by Theorem~\ref{thm:hamilton}, $R_{B}$ contains a
Hamilton cycle $H_B=(V_1'',\dots,V_{k_B}'')$.
 We aim to apply the Blow-up Lemma (Lemma~\ref{lemma:blow_up}) to embed the tree. For condition~(C6) of the lemma, 
the regular pairs must have density close to $d$ in terms of~$\eps_B$. In order to achieve this, we apply Fact~\ref{fact:d_close}
to all regular pairs with density exceeding $d_B$. This gives a spanning subgraph $D^{**}_{B'} \subseteq D^{*}_{B'}$ where
for all~${(V_i,V_j) \in A(R_B)}$, the pair $(V_i,V_j)$ is $(\eps_B/8,+)$-regular with density $d_B \pm \eps_B/16$ in~$D^{**}_{B'}$.

Now, Fact~\ref{fact:super_regular} allows us to move $\ceil{\frac{\eps_B}{2} |V_i''|}$ vertices from each cluster
$V_i''$ to $V_0''$ to obtain a new partition $\{V_0',\dots,V_{k_B}'\}$ of $B'$.
Next, we redistribute $\ceil{\eps_B|B'|} - |V_0'|$ vertices
from~${V_1' \cup \dots \cup V_{k_B}'}$ into $V_0'$
in a balanced way to obtain a new partition $\{V_0,\dots,V_{k_B}\}$
of $B'$ such that $|V_0| = \ceil{\eps_B|B'|}$,
and~${|V_i| = |V_j|}$, for
all~${1 \leq i < j \leq k_B}$. 
By the 
Slicing Lemma~(Fact~\ref{fact:slicing}), 
each pair
$(V_i,V_{i+1})$ is $(\eps_B,d)$-super-regular, and 
each pair~$(V_i,V_j)$ is $(\eps_B,+)$-regular
with density either~0 or~${d \pm \eps_B}$,
for all~${1 \leq i < j \leq k_B}$.

We now apply the Diregularity Lemma (Lemma~\ref{lemma:diregularity}) 
to $D[A' \cup V_0]$ 
with parameters $\eps = \eps_A/16$, $m_0 = m_A$,
and~${d_A = d + 15\eps_A/16}$ to obtain a regularised graph $D^{*}_{A'\cup V_0}$ with
a partition of size $k_A \geq m_A$.
Let $R_{A}^{*}$ be the 
corresponding $(\eps_A/16,d_A)$-reduced graph, which by Lemma~\ref{lemma:reduced} has a spanning oriented subgraph $R_A \subseteq R_A^{*}$
that is a robust~$(\nu'/24,2\tau')$-expander. By Theorem~\ref{thm:hamilton}, $R_A$ has a directed Hamilton cycle~$H_A=(U_1'',\dots,U_{k_A}'')$.
By  Fact~\ref{fact:d_close}, $D^*_{A'\cup V_0}$ has a spanning subgraph $D^{**}_{A'\cup V_0}$ such that
all pairs $(U_i,U_j) \in A(R_A)$ are $(\eps_A/8,+)$-regular  with density~$d_A \pm \eps_A/16$ in~$D^{**}_{A'\cup V_0}$.
By Fact~\ref{fact:super_regular},  moving $\ceil{\frac{\eps_A}{2}|U_i'|}$ vertices
from each cluster we obtain a new partition $\{U_0',\dots,U_{k_A}'\}$ of~${A'\cup V_0}$ such that
each pair~${(U'_i,U'_{i+1})}$ is $(\eps_A/4,d_A-13\eps_A/16)$-super-regular and all pairs~${(U_i',U_j')}$ are~$(\eps_A/4, +)$-regular
with density either $0$ or~${d_A \pm 13\eps_A/16}$.
Note that
\begin{equation}
    \label{eq:u_0}
    |U_0'| \leq \frac{3\eps_A}{4}|A'\cup V_0|.   
\end{equation}

Therefore, by~\eqref{eq:mu_p}, we obtain
\begin{align*}
    |A'\cup V_0| - |U_0'| \geq  \Big(1-\frac{3\eps_A}{4}\Big)|A'\cup V_0|
                         = \Big(1-\frac{3\eps_A}{4}\Big)(\ceil{\mu' n} + \ceil{\eps_B|B'|})
                         \geq |T_A|.
\end{align*}

On the other hand,
\begin{align*}
    |A'\cup V_0| - |U_0'| \leq  |A'\cup V_0|\le  |T_A| + \frac{3\eps_A}{4}|A' \cup V_0|.
\end{align*}

Thus, we can redistribute at most $\frac{3\eps_A}{4}|A' \cup V_0|$ vertices from
$U_1'\cup \dots \cup U_{k_A}'$ to $U_0'$ 
in a balanced way to obtain a new partition~$\{U_0,\dots,U_{k_A}\}$
of $A' \cup V_0$ such that
\begin{equation}
    \label{eq:n_a}
    |A' \cup V_0| - |U_0| = |T_A|, 
\end{equation}
and $|U_i| = |U_j|$, for all $1 \leq i < j \leq k_A$. 
By the 
Slicing Lemma~(Fact~\ref{fact:slicing}),
each pair~$(U_i,U_j)$ is $(\eps_A,+)$-regular
with density either~0 or ~${d \pm \eps_A}$,
and each pair
$(U_i,U_{i+1})$ is $(\eps_A,d)$-super-regular.
Now, let $$A = (A' \cup V_0) \setminus U_0\text{ and }B = (B' \setminus V_0) \cup U_0.$$
Note that, by~\eqref{eq:n_a}, the set $A$ has exactly the same number of
vertices as $T_A$. Also,
the new partition $\{A,B\}$ of $V(D)$ differs from $\{A',B'\}$
only by a small fraction of vertices.
Additionally, we have not discarded arcs incident to vertices in $U_0$, nor arcs 
between $A$ and $B$. Therefore, the partition $\{A,B\}$ of~$V(D)$ satisfies
the following properties:
\begin{enumerate}[(D1)]
    \item $|A| = |T_A|$,
    \item $|N_G^{\diamond}(v,A)| \geq \frac{\gamma}{4}|A|$, for all $v \in B$ and $\diamond \in \{+,-\}$, 
    \item $|N_G^{\diamond}(v,B)| \geq \frac{\gamma}{4}|B|$, for all $v \in U_0$ and $\diamond \in \{+,-\}$, and
    \item $D[A]$ and $D[B]$ are robust $(\nu'/2,\tau')$-expanders.
\end{enumerate}

Finally, let $D_A$ be the digraph obtained from $D^{**}_{A'\cup V_0}$ by removing all vertices from~$U_0$.
Let~$D_B$ be the digraph obtained from~$D^{**}_{B'}$ by
adding all vertices from~$U_0$, and connecting each vertex in $U_0$ to a vertex in~$B$ if such
an arc exists in $D$. 
Note that, by~\eqref{eq:u_0} and since we moved at most $\frac{3\eps_A}{4}|A\cup V_0|$ vertices
to $U_0'$ in order to obtain $U_0$, we have
 \begin{equation}
   \label{eq:size_u0}
  |U_0| \leq 2\eps_A|A' \cup V_0| \leq 7\eps_A |B|.
 \end{equation}

\subsection{Preliminary assignment.}
\label{ssec:preliminary_assignment}

We will use Proposition~\ref{prop:sr_assignment}
to find a preliminary assignment of the tree~$T_B$ into the reduced graph~$R_B$.
This initial assignment ensures that almost all the
vertices of $T_B$ are well-distributed across the vertices of the reduced graph.
Furthermore, it has the property that some special subtrees of $T_B$, selected
according to the type of~$T$, are well-distributed along the arcs of the
Hamilton cycle $H_B$. This prepares the ground for the
later incorporation of the exceptional vertices into the assignment, the
correction of its imperfections, and the embedding of $T_B$.

More precisely,
based on the type of~$T$,
we define two collections of subtrees~$T_B$, denoted by~$\cP$ and~$\cR$.
The collection $\cP$ will be used in the embedding of $T_B$, while
$\cR$ will be used later to incorporate the
exceptional vertices, and to adjust the initial assignment of~$T_B$.
We first describe how these sets are selected according to the type of $T$,
and then show how to apply Proposition~\ref{prop:sr_assignment} to
obtain the desired preliminary assignment.

\sepline 

\noindent \textit{Leafy trees.} 
 If $T$ is leafy, then we let 
  $\cP$ and $\cR$ consist of a collection of subtrees induced by a set of
leaves and their parents, defined as follows. We say that a leaf in an oriented
tree is an \defi{out-leaf} if it has out-degree equal to one, and an \defi{in-leaf}
if it has in-degree equal to one.
For each~${\diamond \in \{+,-\}}$,
let $L^{\diamond} \subseteq V(T_B)$ be the set
of all $\diamond$-leaves of~$T_B$. 
Let~${\circ \in \{+,-\}}$ be such that $|L^{\circ}|= \max\{ |L^{-}|, |L^{+}|\}$. Consider the set $\mathcal S$ of $2$-vertex subtrees formed by the set $P$ of parents of leaves in~$L^{\circ}$
together with one of their leaves in $L^{\circ}$. Then~${|\mathcal S|\ge\lambda n/(2\ell\Delta)}$. Being a subset of a tree,
the set $P$ has an independent subset $P'$ with $|P'|\ge |P|/2$.
Let~$\mathcal{S}'$ consist of all trees from $\mathcal S$ that meet $P'$. 
Then the leaves in trees of $\mathcal{S}'$ are pairwise at distance
at least~$4$. Moreover, let 
$\cP$ be any set of size $\ceil{|\mathcal S'|/2}$ of~$\mathcal{S}'$, and set  $\cR=\mathcal{S'} \setminus \mathcal P$.
Then, we have
\begin{equation}\label{sizePandR}
  |\cP| \ge |\cR|\geq  \left\lfloor\frac{|\mathcal S'|}{2}\right\rfloor \geq \left\lfloor \frac{|\mathcal S|}{4} \right\rfloor \geq \frac{\lambda n}{8\ell\Delta} - 1 \geq \frac{\lambda n}{16\ell \Delta} \gg \delta n.
\end{equation}

\sepline

\noindent\textit{Bare trees.} Assume that $T$ is bare. In this
case, the sets~$\cP$ and $\cR$ consist of
a collection of vertex-disjoint directed bare paths defined as follows.
Let $p=3\ceil{\frac{8}{\gamma}}k_B \leq \ell/2$. By~(T2),
the tree $T_B$ contains a set of $\lambda n/2\ell$
vertex-disjoint directed bare paths, each on $\ell$ vertices.
Splitting each of these paths into $\floor{\ell/p}$ paths of 
length $p$ and ignoring the first one,
we obtain a set of paths on $p$ vertices that are pairwise
at distance at least $p \geq 4$. Moreover, at least half of these paths 
start with an out- or an in-vertex. We consider only the subpaths that
start with the more frequent of these two types.
We then distribute
these subpaths evenly between the sets $\cP$ and $\cR$, discarding one of them if necessary.
Then
\[
  |\cP| = |\cR| \geq \frac{\lambda n}{4\ell}\Big(\floor{\frac{\ell}{p}}-1\Big) - 1 
  \geq \frac{\lambda n}{10p} \gg \frac{\delta n}{p}.
\]
\sepline

\noindent\textit{Switchy trees.} If $T$ is switchy, then
the set $\cP$ will consist of vertex-disjoint
paths of length~$4$, and the set $\cR$ will consist of subtrees
induced by a set of switches and their neighbours. More precisely, they 
are defined as follows.
By~(T3), the tree $T$ contains a set~$\cP'$ of at least~$3\lambda n/128$ 
paths on $4$ vertices all with the same orientation, and a set $\mathcal R$   
of at least~$\lambda n/4\ell$ switches, all vertex-disjoint.
We delete some of the members of~$\mathcal P'$ to obtain a subset $\cP$ such that each of the elements of $\cP$ has distance at least 4 from any other element. 
We can do this, for instance, by selecting a path $P \in \cP'$ and
removing from $\cP'$ all paths that contain vertices at distance at most $4$ from any vertex of $P$. Then select
another path from the updated set $\cP'$ and repeat this procedure until no further paths can be removed. We denote the resulting set by $\cP$. Note that, since $\Delta(T) \leq \Delta$, at most $4\Delta^4$ paths are removed from $\cP'$ at each step.
This gives
\[
  |\cP| \geq \frac{3\lambda n}{512\Delta^4} \gg \delta n, \qand |\cR| \geq \frac{\lambda n}{4\ell}.
\]

\sepline

\noindent Now that the collections $\cP$ and $\cR$ are defined, our
next goal is to find an assignment of~$T_B$ to~$R_B$. 
By applying Proposition~\ref{proposition:regular-cherry-subgraph} to
the reduced digraph~$R_B$ after removing the arcs of~$H_B$, we obtain a regular spanning subgraph  $C'$ with
the cherry property. 
Notice that $C = C' \cup H_B$ is regular 
and has the cherry property.  Let $(v_1,\dots,v_{|B|})$ be a top-down ordering of $V(T_B)$ such that $v_1 = r_B$, and
let $\cP' \subseteq \cP$ and $\cR' \subseteq \cR$ with
\begin{equation}
  \label{eq:size_pr}
  |\cP'| \geq |\cP| - n^{1/2},\qand |\cR'| \geq |\cR| - n^{1/2},
\end{equation}
obtained by applying Proposition~\ref{prop:sr_assignment} twice, once for~$\cP$ and
once for $\cR$. Now, consider a $(C,H_B,\cP' \cup \cR')$-semi-random assignment $\varphi$ of $T$ into $C$. 
By Proposition~\ref{prop:sr_assignment}, with probability at least $1-5\exp(-n^{\beta})$, the assignment $\phi$
satisfies the following properties for each $i \in [k_B]$:  

\begin{enumerate}
  \item[(A1)] $|\phi^{-1}(V_i) \cap \Root(\cP')| = \frac{|\cP'|}{k_B} \pm n^{1-\beta}$,
  \item[(A2)] $|\phi^{-1}(V_i) \cap \Root(\cR')| = \frac{|\cR'|}{k_B} \pm n^{1-\beta}$, 
  \item[(A3)] $|\phi^{-1}(V_i) \cap V(\cP')| = \frac{|V(\cP')|}{k_B} \pm n^{1-\beta}$,
  \item[(A4)] $|\phi^{-1}(V_i) \cap V(\cR')| = \frac{|V(\cR')|}{k_B} \pm n^{1-\beta}$, and
  \item[(A5)] $|\phi^{-1}(V_i) \cap V(T_B)| = \frac{|V(T_B)|}{k_B} \pm 3n^{1-\beta}$, 
\end{enumerate}

where the last property follows by applying Proposition~\ref{prop:sr_assignment} with $S=V(T_B)\setminus (V(\cP')\cup V(\cR'))$
together with (A3) and (A4).
Hence, since $1-5\exp(-n^{\beta}) > 0$, we can fix a~${(C,H_B,\cP'\cup\cR')}$-semi-random assignment $\phi$ such that (A1)--(A5) hold. Moreover, since~$\phi$ is a semi-random assignment it follows that
\begin{enumerate}
    \item[(A6)] every arc of $\cP' \cup \cR'$ is assigned to an arc of $H_B$.
\end{enumerate}

\subsection{Incorporating the  vertices of $U_0$.}
\label{ssec:exceptional_vertices}

We now modify the assignment $\phi$ from the
last step, to incorporate the  vertices of $U_0$.
At the end of this process, we will have an assignment
$\phi':V(T_B)\rightarrow V(R_B) \cup U_0$, and a
subset~$\cR'' \subseteq \cR'$. 

We start by
setting~${\phi'(v) = \phi(v)}$, for all $v \in V(T_B)$,
and setting~${\cR'' = \cR'}$.
We then incorporate the vertices of~$U_0$ one at a time. For each
of these vertices, we remove one subtree from $\cR''$.
In the following, we describe the incorporation of one exceptional vertex.

Let~${u \in U_0}$. 
By (D3), we have
$|N^{\diamond}(u,B)| \geq \frac{\gamma}{4}|B|$, for each $\diamond \in \{+,-\}$.
Thus, there exists~${i^{+},i^{-} \in [k_B]}$ such that 
\begin{equation}\label{eq:deg_exp} 
  \deg^{+}_{D_B}(u,V_{i^{+}}) \geq \frac{\gamma|V_{i^{+}}|}{8} \qand 
  \deg^{-}_{D_B}(u,V_{i^{-}}) \geq \frac{\gamma|V_{i^{-}}|}{8}.
\end{equation}
The process we use to incorporate $u$ depends on the type of $T$.

\sepline

\noindent\textit{Leafy and switchy trees.} If $T$ is leafy or switchy,
we reassign one leaf or switch of~$T_B$ to~$u$. For this, first assume that the root of each tree in $\cR''$ 
is an in-vertex in this tree.
Choose an arbitrary~${x \in \Root(\cR'')}$
such that~$\phi'$ assigns $x$ to the cluster~$V_{i^{+}}$.
Observe that we can always find such vertex~$x$ 
since by (A2), \eqref{eq:size_u0} and~\eqref{eq:size_pr}
we have
\begin{equation*}
  \label{eq:s2_l}
  \begin{split}
  |(\phi')^{-1}(V_{i^{+}})\cap \Root(\cR'')| &\geq |(\phi')^{-1}(V_{i^{+}}) \cap \Root(\cR')| - |U_0| > 0.
  \end{split}
\end{equation*}
Hence, by (A6) and by the assumption that~$\Root(\cR'')$ consists of in-vertices,
the child~$y$ of~$x$ in $\cR''$ 
is assigned
to~$V_{i^{+}-1}$ by~$\phi$. We reassign~$y$ to~$u$ by
setting $\phi'(y) = u$. 
Observe that, since $y$ is an out-vertex, all its neighbours are assigned to~$V_{i^{+}}$ by $\phi'$,
and therefore, this new assignment $\phi'(y)$ is still valid. We  finish by removing  from $\cR''$ 
the tree whose root is $x$.
The process is analogous when $\Root(\cR'')$ consists of out-vertices, with the difference that we
use $V_{i^{-}}$ instead of~$V_{i^{+}}$.

\sepline

\noindent\textit{Bare trees.} If $T$ is bare, we reassign the
second vertex of some path in $\cR''$ to $u$. Moreover, in order
to maintain the assignment valid, we will also need to modify the
assignment of some other vertices of this path.

First, assume that the
paths in $\cR''$ start with an out-vertex.
We choose a path~${Q=x_0x_1\dots x_p \in \cR''}$,
such that $x_0$ is assigned to  $V_{i^{-}}$.
Such a path $Q$  exists since  by~\eqref{eq:size_u0}
and~\eqref{eq:size_pr}, we have that $|(\phi')^{-1}(V_{i^{-}})\cap \Root(\cR'')| > 0$.

By (A6), and
since~$p$ is a multiple of~$k_B$,  vertex $x_p$ is also
assigned to $V_{i^{-}}$. We reassign some vertices of $Q$
in order
to incorporate~$u$, while preserving the property that $Q$ starts
and ends in $V_{i^{-}}$. We begin by reassigning the
vertex $x_1$ to $u$, and the vertex $x_2$ to $V_{i^{+}}$. 
If the cluster $V_{i^{-}+3}$ belongs to the
out-neighbourhood of $V_{i^{+}}$ then no further reassignment
of~$P$ would be  needed.
However, this is not always the case and 
to address this issue we rely on the notion of skewed-traverses,
introduced in Section~\ref{sec:expansion}. 
Let
\[
  W = \Big((V_{i^{+}},V_{i_1}), (V_{i_1-1}, V_{i_2}),\dots, (V_{i_t-1},V_{i^{-}+3})\Big)
\]
be a $(V_{i^{+}}, V_{i^{-}+3})$-skewed-traverse
of length $t \leq \lceil 2/\nu' \rceil+1$ in $R_B$, which
exists by Lemma~\ref{lemma:skewed}. We  reassign the
vertices $x_3,\dots,x_{3+tk_B-1}$ so that each cluster
receives the same number of these vertices, and so that  vertex $x_{3+tk_B-1}$
is assigned to the cluster $V_{i_t-1}$, which belongs to the in-neighbourhood
of~$V_{i^{-}+3}$. To do so, for each $j \in [t]$, we reassign the
vertices $x_{3+(j-1)k_B},\dots,x_{3+jk_B-1}$ to the Hamilton path
induced by $V_{i_j},V_{i_j+1},\dots,V_{i_{j}-1}$.
Since $W$ exists, for each $j \in [t-1]$, we
know that there is an arc
from the cluster $V_{i_j-1}$ assigned to~$x_{3+jk_B -1}$ 
to the cluster $V_{i_{j+1}}$ assigned to~$x_{3+jk_B}$.
For an illustration, see Figure~\ref{fig:incorp_bare}.
This 
reassignment preserves
the number of vertices
assigned to each cluster, except for~$V_{i^{+}}$, which
now has one additional vertex,
and the clusters $V_{i^{-}+1}$, and possibly~$V_{i^{-}+2}$,
each of which now has one fewer vertex.
Furthermore, since the vertex~$x_{3+tk_B-1}$ is
assigned to a cluster in the in-neighbourhood of~$V_{i^{-}+3}$, the new
assignment is valid since it also has the property that~$Q$ starts
and ends in~$V_{i^{-}}$.

\begin{figure}[ht]
    \centering
    \input{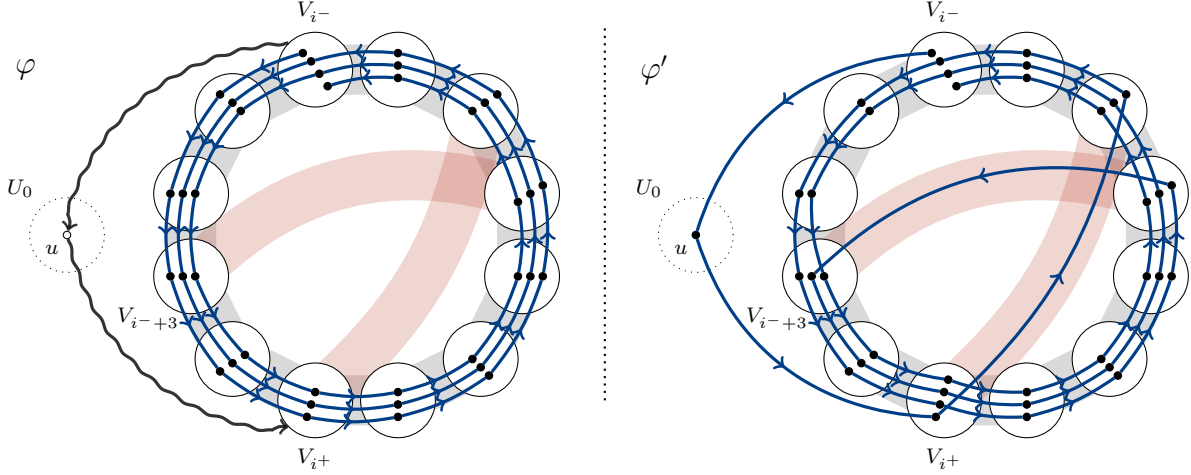}
    \caption{Incorporation of $u$ when $T$ is bare.}
    \label{fig:incorp_bare}
\end{figure}

The process is analogous when the paths in $\cR''$ start with an in-vertex.
The difference is that we select a path starting in $V_{i^{+}}$, reassign
the vertices $x_{3},\dots,x_{3+tk_B-1}$
using a $(V_{i^{+}+3}, V_{i^{-}})$-skewed traverse, and
wind backwards around~$H_B$.

\sepline

By repeating this procedure for each $u \in U_0$, we obtain
a new assignment
\[
  \phi':V(T_B) \rightarrow V(R_B) \cup U_0,
\]
and also a new collection of subtrees $\cR'' \subseteq \cR'$.
As we removed only one subtree from $\cR'$ for
each vertex in $U_0$, and  by~\eqref{eq:size_u0}, we have
\begin{equation}
  \label{eq:size_rpp}
  |\cR''| \geq |\cR'| - n^{1/2} - |U_0| \gg \eps_A n_B.
\end{equation}
In all cases, the assignments of~$\cP'$, and~$\cR''$ were not changed, and for each vertex
in~$U_0$ the reassignment  changed the number of vertices assigned to a cluster by at most one. 
 Thus, combining~\eqref{eq:size_u0} with the fact that~${3n^{1-\beta} \leq \eps_An}$,
which holds because $1/n \ll \beta \ll \eps_A$
yields that for all~${i \in [k_B]}$ we have
\begin{enumerate}
  \item[(A1')] $|(\phi')^{-1}(V_i)\cap \Root(\cP')| = \frac{|\cP'|}{k_B} \pm n^{1-\beta}$,
  \item[(A2')] $|(\phi')^{-1}(V_i)\cap \Root(\cR'')| = \frac{|\cR''|}{k_B} \pm n^{1-\beta}$,
  \item[(A3')] $|(\phi')^{-1}(V_i)\cap V(\cP')| = \frac{|V(\cP')|}{k_B} \pm n^{1-\beta}$,
  \item[(A4')] $|(\phi')^{-1}(V_i)\cap V(\cR'')| = \frac{|V(\cR'')|}{k_B} \pm n^{1-\beta}$, 
  \item[(A5')] $|(\phi')^{-1}(V_i)\cap V(T_B)| = \frac{|V(T_B)|}{k_B} \pm 8\eps_A n_B$, and
  \item[(A6')] every arc of $\cP' \cup \cR''$ is assigned to an arc of $H_B$.
\end{enumerate}

\subsection{Obtaining a balanced assignment.}
\label{ssec:balanced_assignment}

We now adjust the assignment $\phi'$ to obtain a perfectly balanced
assignment. For this, we will reassign some vertices from the
subtrees of~$\cR''$. At the end of this process, we will have obtained a new
assignment~${\phi'':V(T_B) \rightarrow V(R_B) \cup U_0}$ such that the
vertices of~$T_B$ are perfectly balanced across~${V(R_B) \cup U_0}$.

We start by setting~${\phi''(v) = \phi'(v)}$,
for every $v \in V(T_B)$, and $\cR^* = \cR''$. Then,
we choose clusters~$V_{i^{>}}$ and~$V_{i^{<}}$ such
that 
\[
  |(\phi'')^{-1}(V_{i^{>}})| > |V_{i^{>}}| \qand |(\phi'')^{-1}(V_{i^{<}})| < |V_{i^{<}}|.
\]
These clusters exist unless the assignment $\phi''$ is already perfectly balanced (in which case we stop). Next,
we describe how to decrease the number of vertices assigned
to~$V_{i^{>}}$ by one and increase the number assigned
to~$V_{i^{<}}$ by one, without affecting the number of
vertices assigned to the other clusters.
Throughout the process, we maintain the property that at most
one vertex from $V(\cR^*)$ in each cluster is reassigned.

The process to
improve the assignment depends on the type of $T$.

\sepline

\noindent\textit{Leafy and switchy trees.} If $T$ is leafy or switchy,
we reassign certain leaves or switches to adjust
the assignment. Suppose first that the root of each tree in $\cR^*$ is an out-vertex in this tree.
If there exists an arc
from $V_{i^{>}-1}$ to $V_{i^{<}}$, we can simply reassign
one out-leaf or out-switch vertex assigned from $V_{i^{>}}$
to $V_{i^{<}}$ (and leave all other assignment as they are). However, this arc may not exist and therefore, we again use skewed-traverses. Let
\[
  W_1 = \Big((V_{i^{>}-1},V_{i_1}), (V_{i_1-1}, V_{i_2}),\dots, (V_{i_t-1},V_{i^{<}})\Big)
\]
be a $(V_{i^{>}-1},V_{i^{<}})$-skewed-traverse of
length $t \leq \lceil 2/\nu' \rceil+1$,
which exists by Lemma~\ref{lemma:skewed}. For convenience, let $i_0 = i^{>}-1$,
and $i_{t+1} = i^{<}$. Then, for each~${j \in [t]}$, we select one
vertex~${v \in \Root(\cR^*)}$ which is assigned to $V_{i_{j-1}}$ by $\phi''$.
Observe that 
such a vertex $v$ always exists. Indeed, by~\eqref{eq:size_rpp}, and since we reassign only one
vertex from~${(\varphi'')^{-1}(V_{i_{j-1}})\cap\Root(\cR^*)}$, 
for each pair of indices $(i^{>},i^{<})$, and the number of such pairs is bounded above
by $8\eps_A k_B n$, we have $|(\phi'')^{-1}(V_{i_{j-1}})\cap \Root(\cR^*)| > 0$.

By (A6'), the child $y$ of $v$ is assigned
to $V_{i_{j-1}+1}$ by $\phi''$. 
Hence, by setting $\phi''(y) = V_{i_j}$
we decrease the number of vertices assigned
to $V_{i_{j-1}+1}$ by one, and increase the number assigned
to $V_{i_j}$ by one. Moreover, we remove from $\cR^*$ the subtree whose
root is $x$. 
For an illustration, see Figure~\ref{fig:placeholder}.
Note that at the end,  in each cluster we reassigned at most one vertex
from $V(\cR^*)$.
Furthermore, the number of vertices
assigned to $V_{i^{>}}$ decreased by one and the number of
vertices assigned to $V_{i^{<}}$ increased by one.

The procedure is analogous if the root of each tree in $\cR''$ is an in-vertex
in this tree.

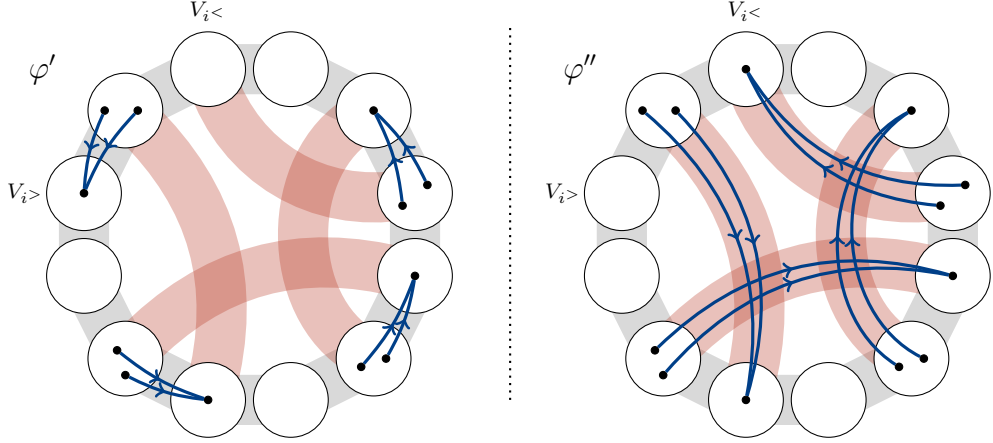
\begin{figure}[ht]
    \centering
    \begin{tikzpicture}[scale=0.55]

    \begin{scope}
        \node  (1aux) at (7, 10) {};
        \node  (2aux) at (5, 9) {};
        \node  (3aux) at (4, 7) {};
        \node  (4aux) at (4, 5) {};
        \node  (5aux) at (5, 3) {};
        \node  (6aux) at (7, 2) {};
        \node  (7aux) at (9, 2) {};
        \node  (9aux) at (11, 3) {};
        \node  (10aux) at (12, 5) {};
        \node  (11aux) at (12, 7) {};
        \node  (12aux) at (11, 9) {};
        \node  (8aux) at (9, 10) {};

        \draw [style=edge superregular] (1aux) to (2aux);
		\draw [style=edge superregular] (2aux) to (3aux);
		\draw [style=edge superregular] (3aux) to (4aux);
		\draw [style=edge superregular] (4aux) to (5aux);
		\draw [style=edge superregular] (5aux) to (6aux);
		\draw [style=edge superregular] (6aux) to (7aux);
		\draw [style=edge superregular] (7aux) to (9aux);
		\draw [style=edge superregular] (9aux) to (10aux);
		\draw [style=edge superregular] (10aux) to (11aux);
		\draw [style=edge superregular] (11aux) to (12aux);
		\draw [style=edge superregular] (12aux) to (8aux);
		\draw [style=edge superregular] (8aux) to (1aux);

        \draw [style=edge superregular, color=BrickRed, bend left=30] (2aux) to (6aux);
        \draw [style=edge superregular, color=BrickRed, bend left=30] (5aux) to (10aux);
        \draw [style=edge superregular, color=BrickRed, bend left=60] (9aux) to (12aux);
        \draw [style=edge superregular, color=BrickRed, bend left=40] (11aux) to (1aux);

		\node [style=cluster] (1) at (7, 10) {};
        
		\node [style=cluster] (2) at (5, 9) {};
        \node [style=vertex mini] (q21) at (4.5,9) {};
        \node [style=vertex mini] (q22) at (5.3,9) {};

		\node [style=cluster] (3) at (4, 7) {};
        \node [style=vertex mini] (q3) at (4,7) {};

		\node [style=cluster] (4) at (4, 5) {};
  
		\node [style=cluster] (5) at (5, 3) {};
        \node [style=vertex mini] (q51) at (5,2.6) {};
        \node [style=vertex mini] (q52) at (4.8,3.2) {};
        
		\node [style=cluster] (6) at (7, 2) {};
        \node [style=vertex mini] (q6) at (7,2) {};
    
		\node [style=cluster] (7) at (9, 2) {};
            
		\node [style=cluster] (9) at (11, 3) {};
        \node [style=vertex mini] (q91) at (11.3,3) {};
        \node [style=vertex mini] (q92) at (10.7,2.8) {};

		\node [style=cluster] (10) at (12, 5) {};
        \node [style=vertex mini] (q10) at (12,5) {};

		\node [style=cluster] (11) at (12, 7) {};
        \node [style=vertex mini] (q111) at (11.7,6.7) {};
        \node [style=vertex mini] (q112) at (12.3,7.2) {};

		\node [style=cluster] (12) at (11, 9) {};
        \node [style=vertex mini] (q12) at (11,9) {};

		\node [style=cluster] (8) at (9, 10) {};

        \draw [style=edge path,postaction={decorate},bend right=10] (q21) to (q3);
        \draw [style=edge path,postaction={decorate},bend right=10] (q22) to (q3);
        \draw [style=edge path,postaction={decorate},bend right=10] (q51) to (q6);
        \draw [style=edge path,postaction={decorate},bend right=10] (q52) to (q6);
        \draw [style=edge path,postaction={decorate},bend right=10] (q91) to (q10);
        \draw [style=edge path,postaction={decorate},bend right=10] (q92) to (q10);
        \draw [style=edge path,postaction={decorate},bend right=10] (q111) to (q12);
        \draw [style=edge path,postaction={decorate},bend right=10] (q112) to (q12);

        \node (varphi) at (3,10) {$\varphi'$};

        \node[left of=3, node distance=22pt] {\tiny $V_{i^{>}}$};
        \node[above of=1, node distance=22pt] {\tiny $V_{i^{<}}$};
    \end{scope}

    \draw[dotted, thick] (14.3, 2) -- (14.3, 11);

    \begin{scope}[xshift=13cm]
        \node  (1aux) at (7, 10) {};
        \node  (2aux) at (5, 9) {};
        \node  (3aux) at (4, 7) {};
        \node  (4aux) at (4, 5) {};
        \node  (5aux) at (5, 3) {};
        \node  (6aux) at (7, 2) {};
        \node  (7aux) at (9, 2) {};
        \node  (9aux) at (11, 3) {};
        \node  (10aux) at (12, 5) {};
        \node  (11aux) at (12, 7) {};
        \node  (12aux) at (11, 9) {};
        \node  (8aux) at (9, 10) {};

        \draw [style=edge superregular] (1aux) to (2aux);
		\draw [style=edge superregular] (2aux) to (3aux);
		\draw [style=edge superregular] (3aux) to (4aux);
		\draw [style=edge superregular] (4aux) to (5aux);
		\draw [style=edge superregular] (5aux) to (6aux);
		\draw [style=edge superregular] (6aux) to (7aux);
		\draw [style=edge superregular] (7aux) to (9aux);
		\draw [style=edge superregular] (9aux) to (10aux);
		\draw [style=edge superregular] (10aux) to (11aux);
		\draw [style=edge superregular] (11aux) to (12aux);
		\draw [style=edge superregular] (12aux) to (8aux);
		\draw [style=edge superregular] (8aux) to (1aux);

        \draw [style=edge superregular, color=BrickRed, bend left=30] (2aux) to (6aux);
        \draw [style=edge superregular, color=BrickRed, bend left=30] (5aux) to (10aux);
        \draw [style=edge superregular, color=BrickRed, bend left=60] (9aux) to (12aux);
        \draw [style=edge superregular, color=BrickRed, bend left=40] (11aux) to (1aux);

		\node [style=cluster] (1) at (7, 10) {};
        \node [style=vertex mini] (q1) at (7,10) {};
        
		\node [style=cluster] (2) at (5, 9) {};
        \node [style=vertex mini] (q21) at (4.5,9) {};
        \node [style=vertex mini] (q22) at (5.3,9) {};

		\node [style=cluster] (3) at (4, 7) {};

		\node [style=cluster] (4) at (4, 5) {};
  
		\node [style=cluster] (5) at (5, 3) {};
        \node [style=vertex mini] (q51) at (5,2.6) {};
        \node [style=vertex mini] (q52) at (4.8,3.2) {};
        
		\node [style=cluster] (6) at (7, 2) {};
        \node [style=vertex mini] (q6) at (7,2) {};
    
		\node [style=cluster] (7) at (9, 2) {};
            
		\node [style=cluster] (9) at (11, 3) {};
        \node [style=vertex mini] (q91) at (11.3,3) {};
        \node [style=vertex mini] (q92) at (10.7,2.8) {};

		\node [style=cluster] (10) at (12, 5) {};
        \node [style=vertex mini] (q10) at (12,5) {};

		\node [style=cluster] (11) at (12, 7) {};
        \node [style=vertex mini] (q111) at (11.7,6.7) {};
        \node [style=vertex mini] (q112) at (12.3,7.2) {};

		\node [style=cluster] (12) at (11, 9) {};
        \node [style=vertex mini] (q12) at (11,9) {};

		\node [style=cluster] (8) at (9, 10) {};

        \draw [style=edge path, bend left=30, postaction={decorate}] (q21) to (q6);
        \draw [style=edge path, bend left=30, postaction={decorate}] (q22) to (q6);
        \draw [style=edge path, bend left=30, postaction={decorate}] (q51) to (q10);
        \draw [style=edge path, bend left=30, postaction={decorate}] (q52) to (q10);
        \draw [style=edge path, bend left=60, postaction={decorate}] (q91) to (q12);
        \draw [style=edge path, bend left=60, postaction={decorate}] (q92) to (q12);
        \draw [style=edge path, bend left=30, postaction={decorate}] (q111) to (q1);
        \draw [style=edge path, bend left=30, postaction={decorate}] (q112) to (q1);
        
        \node (varphi) at (3,10) {$\varphi''$};
        \node[left of=3, node distance=22pt] {\tiny $V_{i^{>}}$};
        \node[above of=1, node distance=22pt] {\tiny $V_{i^{<}}$};
    \end{scope}
\end{tikzpicture}
    \caption{Balancing a switchy tree.}
    \label{fig:placeholder}
\end{figure}

\sepline

\noindent\textit{Bare trees.} If $T$ is bare,
we use the paths in~$\cR^*$ to balance the assignment. First,
assume that the paths in~$\cR^*$ start with an out-vertex. We choose
a vertex $x \in \Root(\cR^*)$ that is the starting vertex of a path $Q=x_0 x_1\dots x_{p-1}\in \cR^*$,
such that $x_0$ is assigned to $V_{i^{>}-1}$. By similar calculations
as in the leafy and neutral case, we can show that such a path $Q$
always exists.
If there is both an arc from $V_{i^{>}-1}$ to $V_{i^{<}}$, and an arc
from $V_{i^{<}}$ to $V_{i^{>}+1}$, then we can reassign
the second vertex $x_1$ of $Q$ (which by (A6') is currently assigned to $V_{i^{>}}$) to $V_{i^{<}}$, thus decreasing the number of vertices assigned to $V_{i^{>}}$ by
one and increasing the number assigned to $V_{i^{<}}$ by one, without
affecting the other clusters.
However, one or both of these arcs may be absent, and so,
we again use skewed-traverses.
Let
\[
  W_1 = \Big((V_{i^{>}-1},V_{i_1}),\dots,(V_{i_{t_1}-1},V_{i^{<}})\Big)\qand  W_2 = \Big((V_{i^{<}},V_{i'_1}),\dots,(V_{i'_{t_2}-1},V_{i^{>}+1})\Big)
\]
be a $(V_{i^{>}-1},V_{i^{<}})$-skewed-traverse of length~$t_1$ and a
$(V_{i^{<}},V_{i^{>}+1})$-skewed-traverse of length~$t_2$,
respectively, with $t_1,t_2 \leq \lceil 2/\nu' \rceil + 1$. 
We start by reassigning the vertices~${x_1,\dots,x_{t_1k_B}}$
in a way that every cluster receives the same amount of these vertices.
Namely, for each $j \in [t_1]$, we reassign the
path~$x_{1+(j-1)k_B},\dots,x_{j k_B}$ to the Hamiltonian path
induced by $V_{i_j},\dots,V_{i_j-1}$. We then assign~$x_{1+t_1k_B}$
to $V_{i^{<}}$. Notice that by doing this we decreased the
number of vertices assigned to $V_{i^{>}}$ and increased the number of
vertices assigned to $V_{i^{<}}$ by one, without affecting the other clusters.
However, this new assignment is not valid. To fix this, we reassign
the vertices~$x_{t_1k_B+1,\dots,(t_1+t_2)k_B}$ as we did before,
but using $W_2$. Then we obtain a valid assignment,
because $x_{(t_1+t_2)k_B+1}$, by (A6'), is assigned to $V_{i^{>}+1}$,
the vertex $x_{(t_1+t_2)k_B}$
is now assigned to~$V_{i'_{t_2}-1}$, and $(V_{i'_{t_2}-1},V_{i+1}) \in W_2$.
Moreover, note that the reassignment of $x_{t_1k_B+1},\dots,x_{(t_1+t_2)k_B}$
does not affect the number of vertices assigned to each cluster.
Therefore, we obtained a new assignment such that the number of vertices
assigned to $V_{i^{>}}$ decreased by one, and the number of vertices
assigned to $V_{i^{<}}$ increased by one.

The procedure is analogous when the paths in $\cR^*$ start with an in-vertex.

\sepline

\noindent By repeating this procedure  until there is no pair of clusters $V_{i^{>}}$
and~$V_{i^{<}}$ with~${|(\phi'')^{-1}\cap V_{i^{>}}| > |V_{i^{>}}|}$
and~${|(\phi'')^{-1}(V_{i^{<}}) \cap V_{i^{<}}| < |V_{i^{<}}|}$, we obtain
an assignment
\[
  \phi'': V(T_B) \rightarrow V(R_B) \cup U_0
\]
such that, for all $i \in [k_B]$, the following holds.
\begin{enumerate}
  \item[(A1'')] $|(\phi'')^{-1}(V_i) \cap \Root(\cP')| = \frac{|\cP'|}{k_B} \pm n^{1-\beta}$, 
  \item[(A2'')] $|(\phi'')^{-1}(V_i) \cap V(\cP')| = \frac{|V(\cP')|}{k_B} \pm n^{1-\beta}$,
  \item[(A3'')] $|(\phi'')^{-1}(V_i)| = |V_i|$, and
  \item[(A4'')] every arc of $\cP'$ is assigned to an arc of $H_B$.
\end{enumerate}
\subsection{Embedding of $T$.}
\label{ssec:embbeding}

Our final goal is to embed $T$ in $D$. First we will embed $T_B$ in~$D_B$, and then
 we  will embed $T_A$ in $D_A$.

\vspace{0.1cm}
\noindent
\textit{Embedding of $T_B$.} We start by setting the stage for the application
of the Blow-up Lemma. Recall that $U_0,V_1,\dots,V_{k_B}$ is a partition
of $V(D_B)$ such that all~$(V_i,V_j)$ are~$\eps_B$-regular with density~$0$ or~$d\pm \eps_B$,
and all $(V_i,V_{i+1})$ are~$(\eps_B,d)$-super-regular.
Also recall that in the incorporation
of $U_0$~(see~\eqref{eq:deg_exp}), for each  $u \in U_0$
we 
chose $V_{i_v^{+}}$ and $V_{i_v^{-}}$ with
\[
  |N^{+}_{G_B}(u,V_{i^{+}})| \geq \frac{\gamma|V_{i^{+}}|}{8} \qand 
  |N^{-}_{G_B}(u,V_{i^{-}})| \geq \frac{\gamma|V_{i^{-}}|}{8}.
\]

We apply the Blow-up Lemma (Lemma~\ref{lemma:blow_up}) with 
\[
  \Delta,\quad
  K_1 =
  K_2 = 1,\quad
  K_3 = 2,\qand
  c = \frac{\gamma}{8}.
\]
The constants $\eps_B$ and $\eps^{*}$ will play the role of~$\eps$ and~$\eps'$.
Note that we have the following hierarchy of constants
\[
  0 < \eps_B \ll \eps^{*} \ll \delta \ll d \ll \Delta^{-1},K_1^{-1},K_2^{-1},K_3^{-1},c.
\]

Let $D_B'$ be the digraph obtained by removing, for each $v \in U_0$,
all arcs incident to $v$ except those going to~$V_{i_v^{+}}$ or coming
from~$V_{i_v^{-}}$. In addition, we remove all arcs between pairs of clusters $(V_i,V_j)$
for which $(V_i,V_j)$ is not an arc of $R_B$.
In the application of the Blow-up Lemma, the underlying
graph of $D_B'$ plays the role of $D$, and the tree $T_B$ plays the role 
of~$H$. The set $U_0$ will correspond to the exceptional set $V_0$.
The partition~$L_0,\dots,L_{k_B}$ of~$V(T_B)$ is the partition induced
by the assignment~$\phi''$.
Furthermore, we set $\psi(v) = \phi''(v)$, for all~${v \in U_0}$,
and~$I = L_0 \cup V(\cP')$.

With the stage prepared, we now verify that conditions (C1)--(C9)
of the Blow-up Lemma (Lemma~\ref{lemma:blow_up}) hold for this setting. By~\eqref{eq:size_u0}, we have
$|L_0| = |U_0| \leq 8\eps_An_B \leq d n_B$,
and thus~(C1) is satisfied. Condition (C2) holds
by the definition of $I$. Conditions (C3), (C6), and (C7) hold by the
construction of $\phi''$. Condition (C4) holds because $\phi''$ assigns
to $U_0$ only vertices of in- or out-degree at most~2.
Therefore, for all $i \in [k_B]$,
\[
  |N_{T_B}(L_0) \cap L_i| \leq 2|L_0| \leq 14\eps_A n_B \leq d|V_i|.
\]
We now verify condition (C5). The choice of the sets $D_i$,
for every $i \in [k_B]$, depends on the type of $T$.

\sepline

\noindent \textit{Leafy trees.} If $T$ is leafy, then for
each $i \in [k_B]$ we select a collection $\cP''_i$ of $\delta|V_i|/2$ subtrees
of~$\cP'$ whose root was assigned to $V_i$. We
set $D = \bigcup_{i \in [k_B]}V(\cP''_i)$, and for each $i \in [k_B]$ we
set~${D_i = D \cap L_i}$. Since $\cP'$ consists of leaves and their parents,
by (A6''), we have $|D_i| = \delta |V_i|$, for every~${i \in [k_B]}$.
So, for every~${i \in [k_B]}$ we have
\[
  |N_{T_B}(D) \cap L_i| = |N_{T_B}(D_{i-1}) \cap L_i| + |N_{T_B}(D_{i+1}) \cap L_i| = \delta|V_i|.
\]
Consequently, 
$\big||N_{T_B}(D) \cap L_i| - |N_{T_B}(D) \cap L_j|\big| = 0 < \eps_B|V_i|$ for every $1 \leq i < j \leq k_B$, and thus~(C5) holds.

\sepline

\noindent \textit{Bare trees.} If $T$ is bare, then, for
each $i \in [k_B]$, we select a collection~$\cP''_i$ of
$\frac{\delta|V_i|}{pk_B}$ subtrees of~$\cP'$ whose root was
assigned to $V_i$.
We set $D = \bigcup_{i \in [k_B]}V(\cP''_i)$,
and, for each $i \in [k_B]$, we set~$D_i = D \cap L_i$. Since every subtree
in $\cP'$ is a directed path of length $p$, and by~(A6'')
these paths wind around the Hamilton cycle $H_B$, it follows
that~${|D_i| = \delta |V_i|}$, for every~${i \in [k_B]}$. Moreover, for
every~${i \in [k_B]}$,
\[
  |N_{T_B}(D) \cap L_i| = |N_{T_B}(D_{i-1})\cap L_i| + |N_{T_B}(D_{i+1})\cap L_i| = \delta|V_i|,
\]
so again, condition (C5) holds.

\sepline

\noindent \textit{Switchy trees.} If $T$ is switchy, then, for
each $i \in [k_B]$, select a collection~$\cP''_i$ of $\frac{\delta|V_i|}{4}$
subtrees of~$\cP'$ whose root was assigned to $V_i$.
We set $D = \bigcup_{i \in [k_B]}V(\cP''_i)$,
and, for each $i \in [k_B]$, we set~${D_i = D \cap L_i}$. Since
all paths in $\cP''$ have the same orientation and length $4$, and
they wind around $H_B$, by (A6''), it follows that $|D_i| = \delta|V_i|$, for
all $i \in [k_B]$. Moreover, they are perfectly distributed over~$H_B$.
Therefore, for every $i \in [k_B]$,
\[
  |N_{T_B}(D) \cap L_i| = |N_{T_B}(D_{i-1})\cap L_i| + |N_{T_B}(D_{i+1})\cap L_i| = \delta|V_i|,
\]
and hence, condition (C5) holds.

\sepline

It only remains to check conditions (C8) and (C9). For   $i \in [k_B]$, let~${E_i \subseteq V_i}$
with $|E_i| = \eps^{*} |V_i|$ be given. Since $|I| = |L_0| + |V(\cP')| \gg \delta n$,
we can choose an arbitrary set~${F_i \subseteq L_i \cap (I \setminus D)}$
consisting of leaves or internal vertices of paths
with $|F_i| = |E_i|$. Let~${\psi_i:E_i \rightarrow F_i}$ be
any bijection. To verify that (C8) holds for this choice, let $v \in E_i$,
and let $j \in [k_B]$ such that~$L_j$ contains a neighbour of $\psi_i(v)$.
By~(A6''), all edges incident to $F_i$ are assigned to 
the Hamilton cycle $H_B$. Therefore,~$(V_i,V_j)$ is
$(\eps_B,d)$-super-regular. Consequently,~${|N_{D_B^{*}}(v)\cap V_j| \geq (d-\eps_B)|V_j|}$, and  (C8) holds.
Finally, let $F = \bigcup_{i\in [k_B]}F_i$. As the vertices in~$F_i$
have degree at most~$2$, it follows that
$|N_{T_B}(F)\cap L_i| \leq 2|E_i| \leq 2\eps^{*}|V_i|$
and so (C9) holds with~${K_3 = 2}$.

Since conditions (C1)--(C9) are satisfied, the Blow-up Lemma  (Lemma~\ref{lemma:blow_up}) gives
a copy of~$T_B$ in~$D_B^{*}$. By the construction of~$\phi''$, this copy
corresponds to an oriented copy of~$T_B$ in~$D_B$.

\vspace{0.1cm}
\subsection{Embedding of $T_A$.} 

To finish the embedding of $T$, we
need to embed $T_A$ into~$D_A$. Let~${v \in V(D_B)}$ be the vertex that hosts 
the root $r_B$ of $T_B$.
As~${|N^{\pm}_G(v,A)| \geq \frac{\gamma}{4}|A|}$, by (D2), there
are $U_{i_v^{+}}$ and $U_{i_{v^{-}}}$ with
\[
  |N^{+}(v,U_{i_v^{+}})| \geq \frac{\gamma|U_{i_v^{+}}|}{8} \qand |N^{-}(v,U_{i_v^{-}})| \geq \frac{\gamma |U_{i_v^{-}}|}{8}.
\]
If $r_A$ is an in-neighbour of $r_B$, we find a preliminary
assignment ${\phi_A:V(T_A)\rightarrow V(R_A)}$ as we did
for $T_B$, with the additional requirement that $\phi_A(r_A) = U_{i_v^{-}}$.
Otherwise, i.e.~if $r_B$ is in an out-neighbour  of $r_B$, we
ensure that~$\phi_A(r_A) = U_{i_v^{+}}$. Note that in either case
 the desired assignment exists because
the root of $T_A$ was assigned uniformly
at random in a cluster of $R_A$ (see Definition~\ref{def:semi_random} (2)).  
Since~$D_A$ has no exceptional vertices, we can skip the incorporation
of exceptional vertices step and obtain a perfectly balanced assignment,
just as we did for~$T_B$. This is possible because the total imbalance
of the preliminary assignment $\varphi_A$ is $\sum_{i \in [k_A]} \big||\phi_A^{-1}(U_i)|-|U_i|\big| \leq k_A n^{1-\beta} \ll \eps_A n_A$.
Now, let $D_A'$ be the digraph obtained from $D_A$ by adding
$v$ and all the arcs from $v$ to $U_{i^{+}}$ and from $U_{i^{-}}$ to $v$. Moreover,
we remove all arcs between pairs of clusters $(U_i,U_j)$ for which $(U_i,U_j)$ is not an arc of $R_A$.
We then finish by
applying the Blow-up Lemma with the underlying graph of $D_A'$ playing the role of~$D$
and treating the root $r_B$ as an exceptional vertex, i.e. we set $L_0 = \{r_B\}$ and
$\psi(r_B) = v$.
Hence, we find a copy of $T_A$ in $D_A$, thus completing the embedding
of~$T$ into~$D$.

\section*{Acknowledgements}

The second author would like to thank Tássio Naia for helpful discussions at the early stages of this project.

\end{document}